\documentclass[12pt]{amsart}

\usepackage{amsmath,amssymb,amsthm}
\usepackage{graphicx}
\usepackage{CJKutf8}
\usepackage{enumerate}
\usepackage[dvips]{color}
\usepackage{version}

\DeclareFontFamily{OT1}{rsfs}{}
\DeclareFontShape{OT1}{rsfs}{n}{it}{<-> rsfs10}{}
\DeclareMathAlphabet{\curly}{OT1}{rsfs}{n}{it}

\newtheorem{Thm}{Theorem}[section]
\newtheorem{Lem}[Thm]{Lemma}
\newtheorem{Fac}[Thm]{Fact}
\newtheorem{Cor}[Thm]{Corollary}

\newtheorem{Prop}[Thm]{Proposition}
\newtheorem{DefProp}[Thm]{Definition-Proposition}
\newtheorem{Conj}[Thm]{Conjecture}
\newtheorem{``Conj"}[Thm]{``Conjecture"}

\newtheorem{Claim}[Thm]{Claim}

\theoremstyle{remark}
\newtheorem{Rem}[Thm]{Remark}
\newtheorem{Ex}[Thm]{Example}
\theoremstyle{definition}
\newtheorem{Def}[Thm]{Definition}

\newtheorem{Step}{Step}

\newtheorem*{ack}{Acknowledgments}
\newtheorem*{fund}{Funding}

\newtheorem{ntt}{Notations}

\newcommand{\iso}{\xrightarrow{   \,\smash{\raisebox{-0.40ex}{\ensuremath{\scriptstyle\simeq}}}\,}}

\excludeversion{memo}

\begin{document}

\title[Tropical geometric moduli compactification of $A_g$]
{Tropical Geometric Compactification of moduli, II \\ 
- $A_g$ case and holomorphic limits - }
\author{Yuji Odaka}

\date{April 23, 2017}

\dedicatory{To the memory of Kentaro Nagao}

\maketitle
\thispagestyle{empty}

\begin{abstract}
We compactify the 
classical moduli variety $A_g$ 
of principally polarized abelian varieties of complex dimension $g$ by 
attaching the moduli of flat tori of real dimensions 
at most  $g$ in an explicit manner. 
Equivalently, we \textit{explicitly} determine the Gromov-Hausdorff limits of 
 principally polarized abelian varieties. 
This work is analogous to \cite{Od.Mg}, 
which compactified the moduli of curves by 
attaching the moduli of metrized graphs. 

Then, we also explicitly specify the Gromov-Hausdorff limits along 
\textit{holomorphic family} of  abelian varieties 
and show that they form special non-trivial 
subsets of the whole boundary. 
We also do it for algebraic curves case and 
observe a crucial difference with the case of abelian varieties. 
\end{abstract}

\tableofcontents


\section{Introduction}

This paper is a companion paper to (or sequel of) \cite{Od.Mg}, which gave a 
couple of compactifications of the moduli of hyperbolic 
projective curves $M_{g}$ and analyzed them. We work on the 
moduli space $A_{g}$ of $g$-dimensional principally polarized 
complex abelian varieties in this paper. 
What we first prove in our \S2 will be 
roughly summarized as follows. 

\begin{Thm}[{cf., \ref{AV.lim.max}, \ref{AV.lim.gen}, \ref{Ag.cptf}}]
The moduli space $A_{g}$ of principally polarized abelian varieties 
over $\mathbb{C}$ with the complex analytic topology admits a  compactification 
$\bar{A_{g}}^{T}$ which attach (as its boundary) the moduli space of 
all real flat tori of dimension $1$ up to $g$, 
the half of the original abelian varieties' real dimension. 
\end{Thm}

We interpret the real flat tori appearing here as tropical abelian varieties, 
as commonly defined so these days (cf., e.g., \cite{BMV},\cite{MZ}), 
so we would like to call this compactification $\bar{A_{g}}^{T}$ 
the \textit{tropical geometric compactification} of $A_{g}$. However, 
in this particular case of abelian varieties, in reality 
the compactification is nothing but the ``Gromov-Hausdorff compactication'' 
i.e., attaching the moduli space of all possible Gromov-Hausdorff limits as metric 
spaces, to the original moduli space ($A_{g}$) as a boundary. 
The \textit{abstract} 
existence of such Gromov-Hausdorff compactification of $A_{g}$ 
(without knowing what are the limits and the structure of the boundary) 
itself is a direct corollary of the well-known Gromov's precompactness 
theorem \cite{Grom}. Our point is to study 
the very explicit structures of the compactification by in particular 
identifying the Gromov-Hausdorff limits as specific flat tori, 
and also show relations with other fields. 

\begin{Rem}
The reason why we would like to allow two possible 
names for the identical compactification above, 
tropical geometric compactifications and Gromov-Hausdorff compactifications,  
is as follows. Basically the author sees 
this coincidence as a special phenomenon only for abelian varieties etc. 
As we also did for curves case (\cite{Od.Mg}), 
the \textit{Gromov-Hausdorff compactification} 
has a natural uniform definition independent of class of 
K\"ahler-Einstein polarized varieties, 
and the boundaries parametrize only metric spaces. 
On the other hand, what we would like to propose and name as 
\textit{tropical geometric compactifications} of the 
moduli spaces of collapsing K\"ahler-Einstein polarized varieties, 
should in general parametrize a priori more 
informations, such as affine structures to regard it as 
`` polarized tropical varieties''
rather than simply \textit{metric spaces} at the boundary. 
In that way, the author believes, 
the compactifications should become ``nicer'' 
with more structures. At the moment of writing this, we only have 
case by case ``working definitions'' 
of such compactifications for $M_{g}$ (\cite{Od.Mg}), 
$A_{g}$ and K3 surfaces (\cite{OO}). See remark \ref{tori.aff.str} for 
more explanation for our particular case. 
\end{Rem}

More precisely, the point of \S2 of this paper is, 
by using the Siegel reduction and so on, to 
\textit{explicitly} determine the Gromov-Hausdorff 
limits as well as the structure of the compatification. 
Then, we go on to 
make some more basic analysis of the compactification including the relations 
of cohomologies and homologies. Then, in \S3 in turn, 
we determine the \textit{holomorphic limits} and compare. 
That is, we consider  Gromov-Hausdorff limits of 
an arbitraliry given punctured holomorphic family of either 
principally polarized abelian varieties of the form $(\mathcal{X},\mathcal{L})\to \Delta^*$ 
or canonically polarized curves 
where $\Delta^*$ denotes a punctured 
smooth algebraic curve $\Delta\setminus \{0\}$. 
More precisely, we take a sequence $t_i\in \Delta^{*}$ 
$(i=1,2,\cdots)$ which converges to the puncture 
$0\in \Delta$ 
and discuss the Gromov-Hausdorff limit of $\mathcal{X}_{t_i}$. The result 
in particular shows that it does \textit{not} depend on the sequence we take, 
once we fix the family. 

Our results in \S3 can be roughly summarized as follows. 
For (\ref{Intro.Ag.alg.lim}), we only prove under some condition (triviality of the 
Raynaud extension) and fully extend in our forthcoming joint paper 
\cite{OO} with Y.Oshima. 

\begin{Thm}
\begin{enumerate}

\item (cf., \ref{Ag.GH}, also \cite{OO}) \label{Intro.Ag.alg.lim}

Given a punctured algebraic family of $g$-dimensional 
principally polarized abelian varieties 
$(\mathcal{X},\mathcal{L})\to \Delta^{*}\ni t$, 
the Gromov-Hausdorff limit at $0\in \Delta$ 
does not depend on the choice of sequence 
converging to $0\in \Delta$ 
and such limits form the union of $A_{g}$ and 
a dense subset (consists of 
``rational points'') inside the whole boundary $\partial \bar{A_{g}}^{T}$. 

\item (cf., \ref{Wolp}, \ref{val.cri.curve}) \label{Intro.Mg.alg.lim} 

Given a punctured algebraic family of smooth projective curves of genus $g\ge 2$ 
$(\mathcal{X},\mathcal{L})\to \Delta^{*}\ni t$, 
the Gromov-Hausdorff limit at the puncture $0\in \Delta$ 
does not depend on the choice of sequence 
converging to $0$ and such limits form the union of $M_{g}$ and 
a finite subset inside the whole real $3g-4$ dimensional 
boundary $\partial \bar{M_{g}}^{T}$. 
\end{enumerate}
\end{Thm}

Finally, our appendix discusses the Morgan-Shalen compactification \cite{MS}, 
recently revisited and extended by Favre \cite{Fav} and Boucksom-Jonsson 
\cite{BJ}. We deal with slight more extensions and prove basic properties as 
to identify our compactifications (in e.g., \ref{mg.hyb.tgc}, \cite{OO}). 
This appendix is used for stating Theorem \ref{mg.hyb.tgc} but since it is 
fairly independent study, we decided to put later as appendix. 
Hence, those interested in \ref{mg.hyb.tgc} could be either once 
skipped to the appendix and then come back to \ref{mg.hyb.tgc} after that 
or simply skip \ref{mg.hyb.tgc} and continue to read the main texts. 

As this series of papers heavily depends on the basic theory of 
Gromov-Hausdorff convergence, we refer to \cite{BBI} if needed. 

\begin{fund}
This work was partially 
supported by the Japan Society for the Promotion of Science 
[Kakenhi, Grant-in-Aid for Young Scientists (B) 26870316] and 
[Kakenhi,   Grant-in-Aid for Scientific Research (S), No. 16H06335]. 
\end{fund}

\begin{ack}
The original version of this preprint was e-print arXiv:1406.7772 appeared in 
June 2014. 
\S2 of this paper is a much revision of the \textit{latter half} of it. 
The former half of the original e-print was put as another 
paper (\textit{v2 of} arXiv:1406.7772) 
also with some mathematical and expository improvements. 
Parts of this work is done during when the author visited Chalmers university and 
Paris. He thanks their warm hospitality and discussions, 
especially R. Berman, S. Boucksom, C. Favre and M. Jonsson, A. Macpherson. 
Since 2014, large part of this paper (and \cite{Od.Mg}, plus some of \cite{OO}) 
has been presented in the talks at various places including 
Kyoto, Tokyo, Oaxaca, Gothenberg, Oxford, Kanazawa, Singapore 
and we appreciate the organizers. 

We would like to dedicate this set of papers (with \cite{Od.Mg}) 
to the heartwarming memory of \textit{Kentaro Nagao}. 

We will further continue our series in a forthcoming joint 
paper with Y.Oshima \cite{OO}, to whom I also thank for 
his helpful comments to this paper as well. \end{ack}


\section{Compactifying $A_g$}\label{Ag.TGC}

\subsection{Gromov-Hausdorff collapse of abelian varieties}

In this section, to each $g$-dimensional principally polarized abelian variety $(V,L)$, 
we associate a \textit{rescaled} K\"{a}hler-Einstein metrics whose \textit{diameters} are $1$.
That is, we consider the flat K\"ahler metric $g_{KE}$ whose K\"ahler class is $c_1(L)$ and 
consider the induced distance on $V$ which we denote by $d_{KE}(V)$ in this paper. 
Then we rescale to $\frac{d_{KE}}{{\it diam}(d_{KE})}$ of diameter $1$, which will be 
the metric in concern. ${\it diam}(-)$ means the diameter. 
Note that in this case, precompactness of the corresponding moduli space $A_{g}$ 
with respect to the associated Gromov-Hausdroff distance 
follows from the famous Gromov's precompactness theorem \cite{Grom} 
(while it also directly follows from our arguments in this section. ) 
We sometimes omit the principal polarization and simply write principally polarized abelian varieties as $V$ or $V_i$ as far as it should not cause any confusion. 
For the basics of the Gromov-Hausdorff convergence in metric geometry, 
we refer to e.g., the textbook \cite{BBI}. 

We proceed to classification of all the possible Gromov-Hausdorff limits of them. 
The author suspects it has been naturally 
expected by experts and at least partially known that 
those collapse should be (real) \textit{flat tori} but unfortunately he could not 
find 
precise study nor results in literatures, so we present here a precise 
statement as well as its proof, and also give explicit determinations of the 
limits. 
In particular, our arguments show that the flat tori which 
can appear as such Gromov-Hausdorff limits, have its real dimension at most $g$ 
(which is the half of the real dimension of the original complex abelian varieties) and 
is characterized only by that condition. 

For simplicity and better presentation of ideas, 
let us first restrict our attention to maximally degenerating case, 
and establish the general case later. 

\begin{Thm}\label{AV.lim.max}
Consider an arbitrary sequence of $g$-dimensional principally polarized complex abelian varieties $\{V_i\}_{i=1,2,3,\cdots}$ which is converging to the cusp $A_0$ of the boundary 
of the Satake-Baily-Borel compactification 
$\bar{A_{g}}^{\text{\begin{CJK}{UTF8}{min}{\scalebox{.7}{\mbox{さ}}}\end{CJK}}BB}$.   
\footnote{The character ``\begin{CJK}{UTF8}{min}{\scalebox{.7}{\mbox{さ}}}\end{CJK}'' is \textit{Hiragana} type character which we pronouce ``SA'', the first syllable of Satake and the idea 
of using this character is after 
Namikawa's book \cite{Nam2} which used \textit{Katakana} ``\begin{CJK}{UTF8}{min}{\scalebox{.7}{\mbox{サ}}}\end{CJK}'' instead 
(but we japaneses rarely use katakana for writing japanese name). The corresponding 
Kanji character \begin{CJK}{UTF8}{min}{\scalebox{.7}{\mbox{佐}}}\end{CJK} is more normal. } 
We denote the flat (K\"ahler) metrics with respect to the polarization 
$d_{KE}(V_i)$ and their diameters ${\it diam}(d_{KE}(V_{i}))$. Then, after passing to an appropriate subsequence, we have a 
Gromov-Hausdorff limit of $\{(V_i, \frac{d_{KE}(V_i)}{{\it diam}(d_{KE}(V_{i}))})\}_i$ which is 
$(g-r)$-dimensional (flat) tori of diameter $1$ with some $(0\le) r(< g)$. 

Conversely, any such flat $(g-r)$-dimensional torus of diameter $1$ with $0\le r< g$ can appear 
as a possible Gromov-Hausdorff limit of such sequence of $g$-dimensional principally polarised abelian varieties 
with fixed diameter $1$. 
\end{Thm}

\begin{proof}

Let us first set up our notations (mainly after \cite{Chai}) on the Siegel upper half space and its compactification theory 
due to I.Satake \cite{Sat}, as in our proof, 
we make essential use of the Siegel reduction theory. 

For a point $Z=X+\sqrt{-1} Y$ of the Siegel upper half space $\mathfrak{H}_{g}$, 
we denote the Jacobi decomposition of $Y$ as $Y={}^{t}BDB$, 
where 
$$B=
\begin{pmatrix}
1 & b_{1,2}         & b_{1,3} & \cdots & b_{1,g} \\
  & 1               & b_{2,3} & \cdots & b_{2,g} \\
  &                 & 1       & \cdots & \vdots  \\
  &\text{\Large{0}} &         & \ddots & \vdots  \\ 
  &                 &         &        &      1  \\ 
\end{pmatrix}, 
$$
$$
D={\it diag}(d_{1},\cdots,d_{g})=
\begin{pmatrix}
d_{1}             &               &       &                 &                   \\
                  & d_{2}         &       & \text{\Large{$0$}}&                      \\
                  &               & d_{3} &                 &                  \\
                  &\text{\Large{$0$}}&       & \ddots          &                 \\ 
                  &               &       &                 &  d_{g}             \\ 
\end{pmatrix}. 
$$
Equivalently, writing $Y={}^{t}\sqrt{Y}\sqrt{Y}$ with a matrix $\sqrt{Y}\in {\it GL}(g,\mathbb{R})$, 
$$\sqrt{Y}={\it diag}(\sqrt{d_{1}},\cdots,\sqrt{d_{g}})B$$ is the corresponding Iwasawa decomposition. 

Using the above notation, 
recall that the Siegel subset $\mathfrak{F}_g(u)$ of 
the Siegel upper half plane $\mathfrak{H}_g$ is defined as 
\begin{equation}\label{S.set.def}
\{ X+\sqrt -1 Y \in \mathfrak{H}_g \mid |x_{ij}|<u, |1-b_{i,j}|<u, 1<ud_1, d_i<ud_{i+1} \text{ for all } i,j  \}. 
\end{equation}
It is known to satisfy 
\begin{equation}\label{S.set}
{\it Sp}_{2g}(\mathbb{Z})\cdot 
\mathfrak{F}_g(u)=\mathfrak{H}_g
\end{equation} for $u\gg 0$. 
Let us set 
$$
\mathfrak{H}_{g}^{*}:=\mathfrak{H}_{g}\sqcup\mathfrak{H}_{g-1}\sqcup\mathfrak{H}_{g-2}\sqcup\cdots \sqcup \mathfrak{H}_{0}. 
$$
\noindent
Then the Satake-Baily-Borel compactification 
$\bar{A_{g}}^{\text{\begin{CJK}{UTF8}{min}{\scalebox{.7}{\mbox{さ}}}\end{CJK}}BB}$ \cite{Sat} 
can be defined as $\mathfrak{H}_g^{*}/\sim $ with some equivalent relation $\sim$ 
extending ${\it Sp}_{2g}(\mathbb{Z})$-action on $\mathfrak{H}_{g}$. 
Thanks to (\ref{S.set}) we can suppose that, fixing sufficiently large $u_{0}\gg 1$, 
\begin{CJK}{UTF8}{min} 
we have $\bar{A_{g}}^{{\scalebox{.7}{\mbox{さ}}}BB}=\mathfrak{F}_{g}^{*}(u_0)/\sim$ with the same equivalent relation, where 
\end{CJK}
$$
\mathfrak{F}_{g}^{*}(u_0):=
\mathfrak{F}_{g}(u_0)\sqcup
\mathfrak{F}_{g-1}(u_0)\sqcup
\cdots
\mathfrak{F}_{0}(u_0) \subset \mathfrak{H}_{g}^{*}. 
$$

\noindent
We refer to \cite{Chai} for the details of above discussions. 
Using the above reduction, we can assume 
that the whole sequence $\{V_{i}\}_i$ of principally polarize abelian varieties are parametrized by 
a sequence $Z_{i}=X_{i}+\sqrt{-1}Y_{i}$ in $\mathfrak{F}_{g}(u_{0})$ for the fixed $u_{0}$ which is sufficiently large. 
Thanks to such boundedness, appropriately passing to a subsequence, we can and do assume 
$X_{i}$ and $B_{i}$ converges. 

As a metric space, 
our principarlly polarized abelian 
variety which corresponds to $Z=X+\sqrt{-1}Y\in \mathfrak{H}_{g}$ is 
$$\mathbb{C}^{g}\Big/
\begin{pmatrix}
1 & X \\
  & Y \\
\end{pmatrix} \mathbb{Z}^{2g}
$$
which is isometric to $\mathbb{R}^{2g}/\mathbb{Z}^{2g}$ with 
metric matrix 
\begin{equation}\label{av.met}
\begin{pmatrix}
1 &   \\
X & Y \\
\end{pmatrix}
\begin{pmatrix}
Y^{-1} &        \\
       & Y^{-1} \\
\end{pmatrix}
\begin{pmatrix}
1 & X \\
  & Y \\
\end{pmatrix}=
\begin{pmatrix}
Y^{-1} &   Y^{-1}X \\ 
XY^{-1}& XY^{-1}X+Y\\ 
\end{pmatrix}. 
\end{equation}

If we denote the corresponding standard basis of the lattice 
$\begin{pmatrix}
1 & X \\
  & Y \\
\end{pmatrix} \mathbb{Z}^{2g}$
\noindent
as $e_1,\cdots,e_2g$ and the K\"ahler-Einstein metric as $g_{KE}$ as before, 
then what we meant by metric matrix is defined as $\{g_{KE}(e_i,e_j)\}_{1\le i,j\le 2g}.$ 

In particular, if $X=(0)$ (\text{zero matrix}), 
then the metric matrix of our torus $\mathbb{R}^{2g}/\mathbb{Z}^{2g}$ is 
$$
\begin{pmatrix}
Y^{-1} &        \\ 
       & Y      \\
\end{pmatrix}. 
$$

Thus what we would like to classify are 
possible Gromov-Hausdorff limits of 
$$\dfrac{1}{{\it diam}(V_{i})}\cdot 
\begin{pmatrix}
Y_{i}^{-1}     &   Y_{i}^{-1}X_i         \\ 
X_{i}Y_{i}^{-1}& X_{i}Y_{i}^{-1}X_{i}+Y_{i}\\ 
\end{pmatrix}. 
$$

We set our notation as follows. 
For our points $Z_{i}=X_{i}+\sqrt -1 Y_{i}$ in the Siegel set 
$\mathcal{F}_{g}(u_{0})$ which give our principally polarized abelian 
varieites $V_{i}$ ($i=1,2,\cdots$), we do the Jacobi decomposition of 
$\sqrt Y_{i}$ and denote the corresponding $d_{j}$s as $d_{j}(V_{i})$. 
Replacing ${\it diam}(V_{i})$ by $d_{g}(V_{i})$, let us first 
classify possible limits of 
$$
\dfrac{1}{d_{g}(V_{i})}\cdot 
\begin{pmatrix}
Y_{i}^{-1}     &   Y_{i}^{-1}X_i           \\ 
X_{i}Y_{i}^{-1}& X_{i}Y_{i}^{-1}X_{i}+Y_{i}\\ 
\end{pmatrix}
$$
instead. Our assumption that it converges to the cusp 
$A_{0}\in \partial 
\bar{A_{g}}^{\text{\begin{CJK}{UTF8}{min}{\scalebox{.7}{\mbox{さ}}}\end{CJK}}BB}$ 
(``maximally degenerating'') is equivalent to that $d_{1}(V_{i})\rightarrow +\infty$ when $i\rightarrow +\infty$ 
from the definition of Satake topology (cf., \cite{Sat}, \cite{Chai}).  

Now, let us set 
$$
r:=\min\Bigl\{(1\le)j(\le g)\mid \displaystyle \liminf_{i\rightarrow +\infty} \frac{d_{j}(V_{i})}{d_{g}(V_{i})}>0\Bigr\}-1. 
$$
Then, after appropriately passing to a subsequence again, we can assume that 
$\frac{d_{r}(V_{i})}{d_{g}(V_{i})}\rightarrow +0$ so that 
$\frac{d_{j}(V_{i})}{d_{g}(V_{i})}\rightarrow +0$ 
for all $j\le r$. 

By once more replacing $\{V_i\}_i$ by a subsequence if necessary, we can assume 
without loss of generality that 
for each $1\le j\le g-r$ 
the sequence $\{\frac{d_{r+j}(V_{i})}{d_{g}(V_{i})}\}_{i}$ converges. 
We denote that convergence values as $a_{r+j}$ for each $j$. 
We prove that then $\{(V_{i},\frac{d_{KE}(V_{i})}{d_{g}(V_{i})})\}_i$ converges 
to a $(g-r)$-dimensional torus as $i\to +\infty$. 

As $d_{g}(V_{i})\rightarrow +\infty$, it follows that 
$Y_{i}^{-1}/d_{g}(V_{i})\rightarrow +0$. On the other hand, 
thanks to our preceded set of processes of replacing $\{V_i\}_i$ by its subsequence, 
the following holds 
$$
\dfrac{1}{d_{g}(V_{i})}
\begin{pmatrix}
d_{1}(V_{i}) &               &              &                 &            \\
             & d_{2}(V_{i})  &              & \text{\Large{$0$}} &            \\
             &               & d_{3}(V_{i}) &                 &            \\
             &\text{\Large{$0$}}&              & \ddots          &            \\ 
             &               &              &                 &d_{g}(V_{i})\\ 
\end{pmatrix} 
$$
$$
\downarrow 
$$
$$
\left(
\begin{array}{ccc|cccc}
     0       &               &                 &                &               &                  &\\
             &     \ddots    &                 &                & \text{\huge{$0$}}&                &\\
             &               &        0        &                &               &                 &\\\hline
             &               &                 &a_{r+1}         &               &                 &\\ 
             &\text{\huge{$0$}}&                &                &  \ddots       &                 &\\
             &               &                 &                &               &  a_{g}=1        &\\
\end{array}
\right)
$$
when $i\rightarrow +\infty$. Please note that the downarrow between the above big matrices  
``$\downarrow$'' means convergence as $g\times g$ real matrices, when $i\to \infty$.  
From the above convergence of the matrices, it follows straightforward that 
$\mathbb{R}^{2g}/\mathbb{Z}^{2g}$ whose metric matrix is $$
\dfrac{1}{d_{g}(V_{i})}
\begin{pmatrix}
Y_{i}^{-1}     &   Y_{i}^{-1}X_i             \\ 
X_{i}Y_{i}^{-1}& X_{i}Y_{i}^{-1}X_{i}+Y_{i}\\ 
\end{pmatrix}
$$
converges to a $(g-r)$-dimensional torus in the Gromov-Hausdorff sense, 
when $i\to +\infty$. From this result, we particularly deduce the following. 

\begin{Claim}
In the above setting, we have 
$$d_{g}(V_{i})\sim d_{KE}(V_{i})$$ i.e., the ratio of the left hand side and the right hand side is bounded on both sides (by some positive constants) when $i \to +\infty$. 
\end{Claim}

Going back to proof of Theorem\ref{AV.lim.max}, 
now we would like to show the other direction i.e., to show that 
\textit{every} $(g-r)$-dimensional flat torus with $0\le r\le g$ of diameter $1$ can indeed 
appear as the above type Gromov-Hausdorff limit. 
Indeed, we can construct such a sequence in the following explicit manner, for instance. 
Fix $(a_{r+1},\cdots,a_{g})\in \mathbb{R}_{>0}^{g-r}$. 
Then set a sequence of $g\times g$ diagonal real matrices $\{D_{i}\}_{i=1,2,\cdots}$ as 

$$D_i:={\it diag}(d_{1,i},\cdots,d_{g,i}):=
\begin{pmatrix}
d_{1,i}&               &              &                 &            \\
             & d_{2,i} &              & \text{\Large{$0$}} &            \\
             &               & d_{3,i} &                 &            \\
             &\text{\Large{$0$}}&              & \ddots          &            \\ 
             &               &              &                 &d_{g,i}\\ 
\end{pmatrix}, $$

\noindent 
where 
$$
d_{j,i}:= 1 
$$
\noindent for  $j\le r$ and 
$$
d_{j,i}:= (i+1)^{r}a_{j}.  
$$
\noindent 
for  $j\ge r+1$. Recall the notation at the beginning of our proof of Theorem \ref{AV.lim.max}. 
Let us fix ``$X$-part and $B$-part", i.e., set $Z_{i}=X_{i}+\sqrt{-1}Y_{i}$ 
with constant $X_{i}=X$ and $B_{i}=B$, where $Y_{i}= ^{t}B_{i}D_{i}B_{i}$ 
is the Jacobi decomposition for $i=1,2,\cdots$ and denote the 
corresponding principally polarized abelian variety as $V_i$ with 
the associated K\"ahler-Einstein metric $g_{KE}(V_i)$. 
Then the Gromov-Hausdorff limit of $(V_i,\frac{g_{KE}(V_i)}{d_g(V_i)})$ for $i\to \infty$ 
is a $(g-r)$-dimensional tori whose metric matrix is 

$$^{t}B  \left(
\begin{array}{ccc|cccc}
     0       &               &                 &                &               &                  &\\
             &     \ddots    &                 &                & \text{\huge{$0$}}&                &\\
             &               &        0        &                &               &                 &\\\hline
             &               &                 &a_{r+1}         &               &                 &\\ 
             &\text{\huge{$0$}}&                &                &  \ddots       &                 &\\
             &               &                 &                &               &  a_{g}=1        &\\
\end{array}
\right)
 B.$$

Letting $B$ runs over all upper trianglar real $g\times g$ matrices, 
we get all $g\times g$ matrices of the form 

$$\left(
\begin{array}{ccc|cccc}
     0       &               &                 &                &               &                  &\\
             &     \ddots    &                 &                & \text{\huge{$0$}}&                &\\
             &               &        0        &                &               &                 &\\\hline
             &               &                 &          &               &                 &\\ 
             &\text{\huge{$0$}}&                &                &     \text{\huge{$P$}}  &                 &\\
             &               &                 &                &               &             &\\
\end{array}
\right)
$$
\noindent 
with a positive definite $(g-r)\times(g-r)$ symmetric matrix $P$, as above limits. 
In such situation, the Gromov-Hausdorff limit of 
$(V_i,\frac{g_{KE}(V_i)}{d_g(V_i)})$ for $i\to \infty$ 
is $(g-r)$-dimensional flat real tori whose metrix matrix is $P$. 
Thus we complete the proof. 
\end{proof}

Please note that $r<g$ can really happen while 
the conjectures \cite[Conjecture1, p.19]{KS}, \cite[Conjecture 5.4]{Gross}, 
motivated by the Strominger-Yau-Zaslow mirror symmetry on Calabi-Yau varieties, 
expect the collapse only to \textit{just half} dimensional affine manifolds (with 
singularities), i.e. $i=g$ case. 
This difference occured naturally, and hence not contradicting, 
simply because we take an arbitrary \textit{sequence} 
rather than dealing with proper algebraic \textit{family} with 
maximal monodromy as they do. 
For a general sequence in $A_g$, we prove the following. 

\begin{Thm}\label{AV.lim.gen}
We use the same notation as Theorem \ref{AV.lim.max}. 
Suppose a sequence of $g$-dimensional principally polarized complex abelian varieties $\{V_{i}\}_{i\ge 1}$ converges to a point of 
$A_c \subset \partial 
\bar{A_{g}}^{\text{\begin{CJK}{UTF8}{min}{\scalebox{.7}{\mbox{さ}}}\end{CJK}}BB}$ 
with $0\le c< g$ in the Satake-Baily-Borel compactification 
$\bar{A_{g}}^{\begin{CJK}{UTF8}{min}{\scalebox{.7}{\mbox{さ}}}
\end{CJK}BB}$. 

Then, after passing to a subsequence, $\bigl(V_i, \frac{d_{KE}(V_i)}{{\it diam}(d_{KE}(V_{i}))}
\bigr)$ 
converges to a $(g-r)$-dimensional 
(flat) tori of diameter $1$ with some $(c\le) r(< g)$, in the Gromov-Hausdorff sense. 

Conversely, any such flat $(g-r)$-dimensional torus of diameter $1$ with $c\le i\le g$ can appear  as a possible Gromov-Hausdorff limit of such sequence of $g$-dimensional principally 
polarized complex abelian varieties with diameter $1$ rescaled K\"{a}hler-Einstein (flat) metrics. 
\end{Thm}

Before going to the proof, let us analyse what the above particularly means. 
Note that the set of possible limits described above is \textit{included} 
in the corresponding limit set of the maximal degeneration case \ref{AV.lim.max}. 
Morally speaking, this can be seen as a special case of more general 
phenomenon that ``degeneration / deformation'' order get reversed once we pass  from  
algebro-geometric setting to 
its tropical analogue. Indeed, similar phenomenon happened in curve case (\cite{Od.Mg}). 

Another very simple fact, which is partially related to above, reflecting such general 
phenomeon is the following. It roughly states that Gromov-Hausdorff limit of degenerating 
spaces sees just ``degenerating part'' and ignores non-degenerating part. 

\begin{Prop}\label{GH.product}
Suppose a sequence of compact metric spaces $\{X^{(i)}\}_{i\in \mathbb{Z}_{>0}}$ decomposes as 
$$X_{1}^{(i)}\times \cdots \times X_{m}^{(i)}$$ as metric spaces 
with $p$-product metric for some $p>0$. 
If the last component $X_{m}^{(i)}$ is ``responsible of degeneration'' 
in the sense that 
\begin{enumerate}
\item ${\it diam}(X_{m}^{(i)})\rightarrow +\infty \text{ and }$ 
\item ${\it diam}(X_{j}^{(i)})\le \text{ constant for all } j\neq m$, 
\end{enumerate}
then the Gromov-Hausdorff limit ``only sees $X_{m}^{(i)}$'' 
in the sense that 
$$
\displaystyle \lim_{i\rightarrow +\infty}(X^{(i)}/{\it diam}(X^{(i)}))=
\displaystyle \lim_{i\rightarrow +\infty}(X_{m}^{(i)}/{\it diam}(X_{m}^{(i)})). 
$$

Here the above ${\it lim}_{i\rightarrow +\infty}$ means the Gromov-Hausdorff limits and 
$(X^{(i)}/{\it diam}(X^{(i)}))$ (resp., $(X_{m}^{(i)}/{\it diam}(X_{m}^{(i)})$) 
means the topological space $X^{(i)}$ (resp., $X_{m}^{(i)}$) with the 
rescaled metric of the original metric, with diameter $1$.  

\end{Prop}

\noindent
A trivial remark is that 
the statement of the above proposition is just equivalent to $m=2$ case but 
we stated as above just to get a better intuition for various applications. 

\begin{proof} The whole point is simply that there is a constant $c$ which satisfies the inequality 
$$
{\it diam}(X_{m}^{(i)})\le {\it diam}(X^{(i)})\le {\it diam}(X_{m}^{(i)})+c
$$
\noindent
for all $i$. The assertion easily follows from the above. 
\end{proof}

Thus indeed if a punctured family of abelian varieties with semiabelian reduction 
with torus rank $(g-r)$ of the central fiber, it follows that the torus part determines 
the Gromov-Hausdorff limit (with fixed diameters). Theorem \ref{AV.lim.gen} is 
reflecting that fact. 

Let us now turn to the proof, 
i.e. the classification of our Gromov-Hausdorff limits of principally polarized abelian varieties. 

\begin{proof}[Proof of Theorem \ref{AV.lim.gen}]
As the proof is a fairly simple extension of 
the proof of maximal degeneration case (\ref{AV.lim.max}), without bringing essentially 
new ideas, here we only sketch the proof, focusing on the differences. 
As in (\ref{AV.lim.max}), thanks to the Siegel reduction theory, 
we can and do fix sufficiently large $u_{0}\gg 0$ so that our sequence can be parametrized by a sequence 
$$\bigl\{Z_{i}=X_{i}+\sqrt{-1}Y_{i}\bigr\}_{i=1,2,\cdots}$$ 
in the Siegel set $\mathfrak{F}_{g}(u_{0})$ (cf., the definition (\ref{S.set.def})). 
Again in the same manner, 
we can and do appropriately take a subsequence so that the following conditions hold. 
\begin{enumerate}
\item $X_{i}$ converges when $i\rightarrow +\infty$, 
\item the upper triangle matrix part $B(V_{i})$ converges when $i\rightarrow +\infty$, 
\item $d_{j}(V_{i})$ for any $(1\le) j(\le) c$ converges when $i\rightarrow +\infty$, 
\item $d_{c+j}(V_{i})\rightarrow +\infty$ when $i\rightarrow +\infty$ for any $(1\le) j (\le (g-c))$. 
\end{enumerate}
Here, the notations are same as the proof of (\ref{AV.lim.max}).
Let us set again 
$$
r:=\max\Bigl\{(1\le)j(\le g)\mid \displaystyle \liminf_{i\rightarrow +\infty} \frac{d_{j}(V_{i})}{d_{g}(V_{i})}=0 \Bigr\}. 
$$
Then in our general case, we have $c\le r\le g$ from the definition of the Satake topology \cite{Sat}. 
The rest of the proof that 
$V_{i}$ converges to a $(g-r)$-dimensional flat torus with diameter $1$ is completely the same. 

Conversely, for a given $r\ge c$, let us prove that 
any $(g-r)$-dimensional torus $T$ with diameter $1$ can appear as the above limit. 
From our (\ref{AV.lim.max}), we know there is a sequence of principally polarized 
abelian varieties $(W_i,M_i)_{i=1,2,\cdots}$ of complex dimension $g-c$. 
Then, we take arbitrary $c$-dimensional principally polarized abelian variety 
$W'$ and set $V_i:=W'\times W_i$ for each $i=1,2,\cdots$ which 
admit natural principal polarizations from the construction. 
Then $\{V_i\}_{i=1,2,\cdots}$ 
with rescaled K\"ahler-Einstein metric 
$\frac{d_{KE}(V_i)}{{\it diam}(d_{KE}(V_i))}$ 
converge to $T$ in the Gromov-Hausdorff sense, as the simple combination of 
Proposition \ref{GH.product} and Theorem \ref{AV.lim.max} show. 


\end{proof}


\subsection{Construction of $\bar{A_{g}}^{T}$ 
and comparison with other tropical moduli space}

Similarly as in the curves case, we rigorously 
define our \textit{tropical geometric compactification} of 
the moduli space of 
principally polarized abelian varieties first set-theoritically as 
$$\bar{A_{g}}^{T}:=A_{g}\sqcup T_g,$$ 
\noindent 
where $T_{g}$ denotes the set (moduli) of real flat tori with diameters $1$ 
whose dimension is $i$ with $1\le i\le g$, from now on. 
\footnote{T of $\bar{A_{g}}^{T}$ stands for \textit{T}ropical while 
T of $T_{g}$ stands for \textit{T}ori. }
Then we put a topology on $\bar{A_{g}}^{T}$ 
 whose open basis can be taken as those of $A_{g}$ 
with respect to the complex analytic topology, and metric balls around point $[T]$ in $\partial \bar{A_{g}}^{T}$ 
$$
B([T],r):=\{[X]\in \bar{A_{g}}^{T} \mid d_{GH}([X],[T])<r\}, 
$$
\noindent
where $d_{dGH}$ denotes, as in \cite{Od.Mg}, 
the Gromov-Hausdorff distance 
with respect to the rescaled metric on each flat torus whose diameter is $1$. 
Then we get the following consequence of Theorem \ref{AV.lim.gen}. 

\begin{Cor}\label{Ag.cptf}
$\bar{A_g}^{T}$ is a compact Hausdorff space containing $A_g$ as an open dense 
subset. 
\end{Cor} 

\begin{Rem}\label{tori.aff.str}
Each $k$-dimensional real flat torus $T$ ($1\le k\le g$) 
that is parametrized at the boundary of 
our $\bar{A_{g}}^{T}$ carries a canonical integral affine structure 
(up to rescaling by positive constants) by 
identifying $T$ with $\mathbb{R}^{k}/\mathbb{Z}^{k}$ and consider the 
corresponding coordinates (or its equivalent class with respect to the 
positive scalar rescaling). This is indeed the integral affine structure we should 
put from the context of the Strominger-Yau-Zaslow mirror symmetry picture 
for the principal polarization case as well (cf., e.g., \cite{GS}). 

Therefore, such affine structure does not give additional structure which 
enlarge the Gromov-Hausdorff compactification of $A_{g}$. It is the reason 
why we simply call the above compactification, the tropical geometric compactification. We refer to the footnote of the introduction
\end{Rem}

Note that if we forget complex structures of principally polarized abelian varieties, 
it gives nontrivial morphism from $A_g$ to a moduli space of certain flat tori of 
real dimension $2g$. The latter moduli space has the same dimension as that of 
$A_g$. 
\footnote{For example, at the locus which parametrizes products of $g$ elliptic curves, 
it gives generally $2^{g}$ to $1$ morphism. } Indeed, discreteness of the fibres of the 
forgetful morphism follows from the fact that, adding marking 
$[\pi_{1}(\text{complex ab. var of dim} g)\iso \mathbb{Z}^{2g}]$, 
which is obviously discrete data, recovers the complex structure of the abelian varieties. 
It easily follows from the fact that 
the metric matrix (\ref{av.met}) has enough information to recover $X$ and $Y$.

Let us clarify the simple relation with the moduli space $A_{g}^{tr}$ of tropical abelian varieties constructured in \cite{BMV}. In their language, 
the boundary of our tropical geometric compactification $\bar{A_{g}}^{T}$ is 
$$
\partial \bar{A_{g}}^{T}\simeq A_{g}^{tr}/\mathbb{R}_{>0}=(\Omega^{rt}\setminus \{0\})/(GL(g,\mathbb{Z})\cdot \mathbb{R}_{>0}), 
$$
where $A_{g}^{tr}$ is the moduli space of $g$-dimensional tropical (principally polarized) abelian varieties $\mathbb{R}^{g}/\Lambda$ 
in the sense of \cite{BMV}, $\Omega^{rt}$ (resp., $\Omega$) is the cone of positive semidefinite forms (resp., positive definite forms) 
on the universal covering $\mathbb{R}^{g}$ whose null space has a basis inside the rational vector space $\Lambda\otimes_{\mathbb{Z}} \mathbb{Q}$, 
following their notations. Note $\Omega\subset \Omega^{rt}\subset \bar{\Omega}$. 

\begin{Rem}
We make a simple observation on the relation between our Gromov-Hausdorff limits of 
principally polarized abelian varieties with the \textit{dual (intersection) complex} (cf., 
\cite{KS}, \cite{Gross}) of algebraic degenerations of them. 
Such connection is natural, after 
the well-known 
conjectures of Kontsevich-Soibelman \cite{KS} and Gross-Siebert (cf., \cite{Gross}) for 
their approach to the Strominger-Yau-Zaslow conjecture \cite{SYZ}. In their studies, 
they also predict and partially establish  
that given a maximal degeneration of general Calabi-Yau manifolds, 
the dual complex of the special fiber is ``close to'' the Gromov-
Hausdorff limit of Ricci-flat metrics with fixed diameters. 

Let us think of the relative compactification of Alexeev and Nakamura 
(\cite{AN},\cite{Ale1},\cite{Nak1},\cite{Nak2}), 
of a semi-abelian reduction of a generically abelian scheme. 
Due to \cite[(3.17)]{AN}, \cite{Ale1}, \cite[(4.9)]{Nak1}, the dual complexes 
\footnote{
also called ``incidence complex'' (cf., e.g. \cite{Tyo12}), ``dual graph'', 
or ``dual intersection complex'' (cf., e.g., \cite{Gross}) etc 
}
are the duals of the Delaunay triangulations of $(g-r)$-dimensional tori 
which are topologically 
of course always real torus of $(g-r)$-dimension. This coincides with our 
collapsed limits (\ref{AV.lim.gen}), except for a slight difference that 
the tori can get lower dimension as we considered an arbitrary 
\textit{sequence} there. We give a closer connection between 
Alexeev-Nakamura type degeneration of abelian varieties and 
our Gromov-Hausdorff limits later in section \ref{Alg.deg.AV}. 
\noindent 
Similarly, for the case of curves \cite{Od.Mg}, the collapsed limit 
concides with the dual graph of the limit stable curves and 
for higher dimensional semi-log-canonical models, we believe the 
collapsed limits along horomorphic one parameter degeneration 
$\mathcal{X}\twoheadrightarrow \Delta_{t}$ (partially analyzed in \cite{Zha}) 
should be at least homeomorphic to the dual complex of lc centers of 
a log crepant blow up $\tilde{\mathcal{X}}_{0}$ of $\mathcal{X}_{0}$ 
whose normalization $\tilde{\mathcal{X}}_{0}^{\nu}$ with the 
conductor divisor ${\it cond}(\nu)$ is a dlt pair. 
\footnote{However, unfortunately such existence is unknown. 
Also cf., \cite[5.22]{Kol16}.} 
We also refer to \cite{BJ} for related recent study.  
(The author morally 
sees this as a variant of the Yau-Tian-Donaldson correspondence and 
wishes to come back to this connection at deeper level in future. )
\end{Rem}


\subsection{Finite and infinite joins of $A_{g}$}\label{Ag.join}

Completely similarly as curve case (\cite{Od.Mg}), 
we can naturally construct \textit{joins} of our tropical geometric compactifications 
$\bar{A_{g}}^{T}$, 
thanks again to the inductive structure of the boundaries (\ref{emb.Tg}). 

\begin{Def}

The \textit{finite join} of our tropical geometric compactifications is 
defined inductively as 
$$\overline{A_{\leq g}}^{T}:=
\overline{A_{\leq (g-1)}}^{T}\cup_{T_{g-1}}\bar{A_{g}}^{T}. $$

The union is obtained via two canonical inclusion maps $T_{g-1}\hookrightarrow T_{g}$ and 
$T_{g-1}\hookrightarrow \overline{A_{\leq (g-1)}}^{T}$. 
We call $\overline{A_{\leq g}}^{T}$ 
a \textit{finite join} of our tropical geometric compactifications. 

From the definition, we have 
$$\cdots\overline{A_{\leq (g-1)}}^{T}\subset 
\overline{A_{\leq g}}^{T}\cdots.$$
Then we set 
$$
\bar{A_{\infty}}^{T}:=
\varinjlim_{g} \overline{A_{\leq g}}^{T}=\cup_{g}\overline{A_{\leq g}}^{T},  
$$
and call it the \textit{infinite join} of our tropical geometric compactifications. 

\end{Def}

The boundary of our infinite join 
$\bar{A_{\infty}}^{T}$ by which we mean the natural locus 
$\cup_{g}
(\partial \bar{A_{g}}^{T}=S_{g}^{wt})\subset \bar{A_{\infty}}^{T}$, should be regarded as a tropical version of 
``$A_{\infty}$''\footnote{They call it ``universal moduli spaces'' of 
abelian varieties} 
introduced and studied recently 
in \cite{JJ} a while after the appearance of the 
first version of this paper. 

Also note $\bar{A_{\infty}}^{T}$ is connected and all our tropical geometric compactification 
$\bar{A_{g}}^{T}$ is inside this infinite join. 
In particular, $A_{g}$ for all $g$ is inside this connected ``big infinite dimensional moduli space''.


\subsection{On the (co)homology groups}


About the open dense locus $A_{g}$, 
the following has been classically known as a result of A.Borel who proved 
through studying the vector spaces of ${\it Sp}_{2g}(\mathbb{R})$-invariant 
different forms 
and the group cohomology interpretation that 

$$
H^{i}(A_{g};\mathbb{Q})=H^{i}({\it Sp}_{2g}(\mathbb{Z});\mathbb{Q}). 
$$

\begin{Thm}[\cite{Bor}]
$H^{i}(A_{g};\mathbb{Q})=\mathbb{Q}[x_{2},x_{6},x_{10}...]|_{weight=i}$ for $0\le i<g-1$,
where the right hand side is a polynomial generated by $x_{4a+2}$ whose weight is $4a+2$. 
In particular, $H^{i}(A_{g};\mathbb{Q})=0$ if $i$ is odd, less than $g-1$ and 
the stable cohomology is naturally $\displaystyle \varinjlim_{g} H^{*}(A_{g})=\mathbb{Q}[x_{2},x_{6},x_{10}...]$. 
\end{Thm}

There are also many studies on the homology of symplectic groups such as \cite{Char}, \cite{MV} etc. 
Using such topological results on $A_{g}$, at least partially the study of (co)homologies of the boundary $T_g$ 
gives some informations on those of $\bar{A_{g}}^{T}$. For instance, a simple observation is that 
${\it dim}(T_g)=3g-4$ combined with the long exact sequence of the Borel-Moore homology groups gives that 
$H_{i}(\bar{A_{g}}^{T};\mathbb{Q})=0$ for 
if $i$ is even and $i> g^{2}$. 






Motivated partially from the above discussion, 
from now on, let us study the boundary $T_g$. 
Note that $T_g$ has the following orbifold as an open dense locus 
$$\Omega/(\mathbb{R}_{>0}\cdot {\it GL}(g,\mathbb{Z})), $$ 
which we will write $T_{g}^{o}$. Then 
$$
T_g=T_{g}^{o}\sqcup T_{g-1}, 
$$
\noindent
so that we can partially study the (co)homology of $T_g$ inductively, 
once we know those of $T_{g}=\Omega/(\mathbb{R}_{>0}\cdot {\it GL}(g,\mathbb{Z}))$. 
However, the author does not know well how this cohomology behaves except for the asymptotic behaviour of the 
lower degree due to A.Borel \cite{Bor}, that is 
$$
H^{i}(T_{g};\mathbb{Q})=H^{i}({\it GL}(g;\mathbb{Z});\mathbb{Q})=\mathbb{Q}[x_{3},x_{5},x_{7},\cdots]|_{weight=i}, 
$$
\noindent
for $i\leq (g-5)/4$. 
Here, $x_i$ has weight $i$. 

As in the discussion of the previous paper \cite{Od.Mg}, 
we have a canonical chain of closed embeddings 

\begin{equation}\label{emb.Tg}
T_{g} \hookrightarrow T_{g+1} \hookrightarrow \cdots 
\end{equation}
which is analogous to the boundary structure of the Satake-Baily-Borel  
compatification of $A_{g}$. 
We have the following asymptotic triviality of the topologies, 
analogous to that of curves case (\cite{Od.Mg}). 

\begin{Prop}
The 
topological space $T_{\infty}$ is contractible. 
${\it Im}(H_{k}(T_{g};\mathbb{Q})\rightarrow H_{k}(T_{g+1};\mathbb{Q}))=0$ for any $k$ and $g$. 
\end{Prop}

\begin{proof}
We imitate the idea of curve case analogue in \cite{Od.Mg} but 
in this abelian varities case, it is even easier. 
However, the whole point is still the same, that is to construct an 
extension $\psi_{g}\colon CT_{g}\rightarrow T_{\infty}$ of the identity map of $T_{g}$ 
where $CT_{g}:=(T_{g}\times [0,1])/(T_{g}\times \{1\})$, which is compatible with lower $\psi$ i.e., 
$\psi_{g}|_{T_{g-1}}=\psi_{g-1}$. 

For $((X,d_{X}), t)\in T_{g}\times [0,1]$ ($d_{X}$ denotes the flat metric on $X$), 
we define 
$$
\psi_{g}(X,t):=\text{ rescale of }((X,(1-t)d_{X})\times S^{1}(t)) \text{ with diameter } 1. 
$$
\noindent 
The continuity of the map is obvious. 
Here, the product means the $2$-product metric (i.e., simply the square root of the sum of squares of direction-wise distances). 
It is straightforward to confirm the requirements of the map. 
\end{proof} 

\noindent
Intuitively speaking, the all $g$-dimensional tori continuously and simultaneously 
change to once $(g+1)$-dimensional tori but later collapse to a circle of circumference $1$. 

On the other hand, we have the following exact sequence from which 
high nontriviality of the topologies of $T_{g}$ follows.

\begin{Prop}

We have the following two long exact sequences. 

\begin{enumerate}
\item{
$$\cdots \to H_{k}({\it GL}(g;\mathbb{Z});\mathbb{Q})^{*}\to H^{k}(T_{g},\mathbb{Q})\to H^{k}(T_{g-1};\mathbb{Q})\to \cdots$$
$$\cdots \to H_{k+1}({\it GL}(g;\mathbb{Z});\mathbb{Q})^{*}\to H^{k+1}(T_{g},\mathbb{Q})\to H^{k+1}(T_{g-1};\mathbb{Q})\to \cdots.$$ 
}
\item{
$$\cdots \to H_{k}(T_{g-1};\mathbb{Q})\to H_{k}(T_{g},\mathbb{Q})\to H^{k}({\it GL}(g;\mathbb{Z});\mathbb{Q})^{*}\to \cdots$$
$$\cdots \to H_{k-1}(T_{g-1};\mathbb{Q})\to H_{k-1}(T_{g},\mathbb{Q})\to H^{k-1}({\it GL}(g;\mathbb{Z});\mathbb{Q})^{*}\to \cdots$$
}
\end{enumerate}
\end{Prop}

\begin{proof}
These are simply the long exact sequences of compactly supported cohomology groups and the  Borel-Moore homology groups respectively, 
combined with Lefschetz duality for orbifold $T_{g}\setminus T_{g-1}=\Omega/{\it GL}(g,\mathbb{Z})$. 
\end{proof}


\subsection{Gromov-Hausdorff limits with other rescaling}

There are of course some other ways of rescaling the metrics of 
abelian varieties which could produce essentially 
different (pointed) Gromov-Hausdorff limits. 
One of the nontrivial rescaling is 
(i) via fixing the \textit{volume} while another is 
(ii) 
via fixing the \textit{injectivity radius}. 
We discuss such two other ways of rescaling but 
before that, let us illustrate the differences by a simple example. 

\subsubsection{A simple example} 
Consider again 
a degenerating sequence of elliptic curves 
$$
E_k:= \mathbb{C}/(\mathbb{Z}+\mathbb{Z}k(a\sqrt{-1}))
$$
for $k=1,2,\cdots$, while $a>1$ fixed. 
In this case, this is maximally degenerating so that the corresponding 
``torus rank" is $r=1=g$. 

The ``diameter fixed" Gromov-Hausdorff limit is $S^1(1/2\pi)$ as we observed. 
Instead if we fix the injectivity radius, then as the metric is standard metric of $\mathbb{C}$ 
we get 
$$(\mathbb{R}/\mathbb{Z})\times (\sqrt{-1}\mathbb{R})$$ 
as the pointed Gromov-Hausdorff limit. 

On the other hand, if we fix the volume of each $E_{k}$, then we rescale the metric by 
multiplying the lengths by $1/\sqrt{ka}$. Then the pointed Gromov-Hausdorff limit is the imaginary axis 
$$(\sqrt{-1}\mathbb{R})\subset \mathbb{C}.$$ 

\noindent
In our Gromov-Hausdorff interpretation of the Satake-Baily-Borel compactification 
$\mathbb{C}\subset\mathbb{CP}^1$ discussed above (\ref{AV.v.fix}), 
this line of infinite length is corresponding to the cusp $\{\infty\}$ while the open part $A_{1}\simeq\mathbb{C}$ parametrizes flat $2$-dimensional tori of volume $1$.

\subsubsection{Fixing the injectivity radius}

In this subsection, we study pointed Gromov-Hausdorff limits 
of $g$-dimensional principally polarized abelian varieties with 
fixed \textit{injectivity radius}, that is morally the ``minimal'' 
non-collapsing limits. 
We keep using the previous notation of this section. 
Recall that for our sequence $\{V_i\}_{i=1,2,\cdots}$ of principally polarized abelian varieties 
of $g$-dimension, the corresponding point in the Siegel set is denoted as 
$Z_{i}=X_{i}+\sqrt{-1}Y_{i}$ with $Y_i={}^{t}B_{i}D_{i}B_{i}$ (the Iwasawa decomposition of 
$\sqrt{Y_i}$). 

Similarly as before, after passing to a subsequence, we can and do 
assume that for some $0\le r<g$, 
\begin{enumerate}
\item both $X_i$ and $B_i$ converge when $i$ tends to infinity, 
\item $d_j(V_i)$ for all $1\le j\le r$ converges to finite value while 
\item \label{cond.3} $d_j(V_i)$ for all $j>r$ (strictly) diverges to infinity when $i$ tends to infinity. 
\end{enumerate}
\noindent
Here, what we meant by the strict divergence in the above (\ref{cond.3}), 
is that all subsequences diverge. 
We assume the above three conditions throughout the rest of present subsection. 

Let us first start with the simplest situation, i.e., those satisfying the following conditions. 
\begin{enumerate}
\item[(iv)] $X_i=0, B_i=I_g$ (unit matrix) , 
\item[(v)] $d_j(V_i)=a_j$ for all $j\le r$ and 
\item[(vi)] $d_j(V_i)=i\cdot a_j$ for all $j>r$. 
\end{enumerate}

The real constants 
$a_1, \cdots , a_g$ above satisfy that  
$$
1<u_{0}a_{0}, a_{i}<u_{0}a_{i+1}. 
$$
Intuitively $g-r$ is the corresponding ``torus rank" of limit. Then from the 
above assertions, it is easy to see that 

\begin{Prop}\label{fix.inj.rad}
The pointed Gromov-Hausdorff limit of the rescaled K\"ahler-Einstein 
metrics on $V_i   (i\to \infty)$ with 
fixed \textit{injectivity radius} $1$ in the above notation is isometric to 
$$\displaystyle \prod_{r< j\le g} S^1\Bigl(\frac{a_{g}}{2\pi a_{j}}
\Bigr)
\times \mathbb{R}^{g+r},$$ 
where $S^1(a)$ denotes a circle with radius $a$.  
\end{Prop}

\noindent
Note that ``pointed" does not cause ambiguity in this situation, 
thanks to the homogenuity of abelian varieties. 

\begin{proof}[Sketch proof]
In the above simple situation, $V_{i}$ with the K\"ahler-Einstein 
metrics on $V_i   (i\to \infty)$ is decomposed as 
$$\prod_{1\le j\le g}(\mathbb{C}/\mathbb{Z}+\mathbb{Z}\sqrt -1 (d_{j}(V_{i}))^{2})$$ 
as K\"ahler manifolds and each 
$\mathbb{C}/\mathbb{Z}+\mathbb{Z}\sqrt -1 (d_{j}(V_{i}))^{2}$ 
is isometric to the product metric space 
$S^{1}(\frac{1}{2\pi (d_{j}(V_{i}))})\times S^{1}(\frac{(d_{j}(V_{i}))^{2}}{2\pi})$. 
It concludes that $V_{i}$ with 
rescaled K\"ahler-Einstein 
metrics on $V_i   (i\to \infty)$ with 
fixed injectivity radius $1$ 
is isometric to 
$$\prod_{1\le j\le g}
\biggl(S^{1}\biggl(\frac{c_{i}}{2\pi d_{j}(V_{i})}\biggr)\times S^{1}
\biggl(\frac{c_{i} d_{j}(V_{i})}{2\pi}\biggr)\biggr),$$ 
for  positive real number $c_{i}$ defined as 

$$c_{i}:=\biggl(\min\biggl\{ \frac{1}{2\pi d_{j}(V_{i})},\frac{d_{j}(V_{i})}{2\pi}
\mid 1\le j\le g  \biggr\}\biggr)^{-1}.$$ 

Then the assertion of Proposition \ref{fix.inj.rad} follows. 
\end{proof}

Note that the limit above does not reflect any abelian part data (``$a_{1},\cdots,a_r$") 
encoded in 
the boudary of 
the Satake-Baily-Borel compactification. 
We prefer the other Gromov-Hausdorff limits, 
hence we do not pursue the above type rescaled limits further, 
partially because 
our main intention is (still) to investigate nice moduli compactifications 
that occur from other rescalings. 

\subsubsection{Fixing the volume} 

We remove the assumptions $(iv), (v), (vi)$ now while keep assuming 
$(i), (ii), (iii)$ and analyse the corresponding Gromov-Hausdorff limits while 
fixing \textit{volume}s in turn. 
Note that to fix the volume of $\{V_i\}$, say as $1$, 
is simply resulting to the metric matrices 
\begin{equation}\label{last.met.mat}
\begin{pmatrix}
Y_{i}^{-1}     &   Y_{i}^{-1}X_{i}             \\ 
X_{i}Y_{i}^{-1}& X_{i}Y_{i}^{-1}X_{i}+Y_{i}\\ 
\end{pmatrix}
\end{equation}
of $\mathbb{R}^{2g}/\mathbb{Z}^{2g}$ without any normalization factor. 
Let $B$ be $\displaystyle \lim_{i\rightarrow \infty} B_i$ and let $X$ be 
$\displaystyle \lim_{i\rightarrow \infty} X_i$. 
We extract the $(r\times r)$ upper left part $X'$ of $X$ and $Y'$ of $Y$ as 
$$
X':=
\begin{pmatrix}
x_{1,1} & x_{1,2}         & x_{1,3} & \cdots & x_{1,r} \\
x_{2,1} & x_{2,2}         & x_{2,3} & \cdots & x_{2,r} \\
\vdots  &  \cdots           &  \cdots    & \cdots & \vdots  \\
\vdots  &  \cdots           &\cdots   & \cdots  & \vdots \\ 
x_{r,1} &x_{r,2}         & x_{r,3} & \cdots & x_{r,r}  \\ 
\end{pmatrix}, 
$$

$$
Y':=
\begin{pmatrix}
y_{1,1} & y_{1,2}         & y_{1,3} & \cdots & y_{1,r} \\
y_{2,1} & y_{2,2}         & y_{2,3} & \cdots & y_{2,r} \\
\vdots  &  \cdots           & \cdots      & \cdots & \vdots  \\
\vdots  &  \cdots           &\cdots   & \cdots  & \vdots \\ 
y_{r,1} &y_{r,2}         & y_{r,3} & \cdots & y_{r,r}  \\ 
\end{pmatrix}, 
$$ 
and denote the $(r\times r)$ upper left part $B'$ of $B$ as 
$$
B':=
\begin{pmatrix}
1 & b_{1,2}         & b_{1,3} & \cdots & b_{1,r} \\
  & 1               & b_{2,3} & \cdots & b_{2,r} \\
  &                 & 1       & \cdots & \vdots  \\
  &\text{\Large{0}}  &         & \ddots & \vdots  \\ 
  &                 &         &        &      1  \\ 
\end{pmatrix}. 
$$
Then our metric matrices (\ref{last.met.mat}) converge to the following except for 
lower right i.e., $(*)$-part of $(g-r)\times (g-r)$. 

$$
\left(
\begin{array}{c|ccc|c|ccc}
\text{\large{$F$}} & 0         & \cdots & 0          & \text{\large{$G$}}   & 0          & \cdots              & 0            \\\hline
0                  & 0         & \cdots & 0          & 0                    & 0          & \cdots              & 0            \\ 
\vdots             & \vdots    & \ddots & \vdots     &  \vdots              & \vdots     & \ddots             & \vdots       \\ 
0                  & 0         & \cdots & 0          & 0                    & 0         & \cdots               & 0            \\\hline 
\text{\large{${}^{t}G$}} & 0   & \cdots & 0          & \text{\large{$H$}}   & 0         & \cdots               & 0            \\\hline
0                  & 0         & \cdots & 0          & 0                    &             &                   &              \\ 
\vdots             & \vdots    & \ddots & \vdots     &  \vdots              &             & \text{\large{*}}  &               \\ 
0                  & 0         & \cdots &      0     & 0                    &             &                   &               \\ 
\end{array}
\right). 
$$

Here, the submatrices $F, G, H$ are those defined by $X,X',Y,Y'$ as 
\begin{itemize}
\item $(Y')^{-1}=F$, 
\item $(Y')^{-1}X=G$ and 
\item $X'(Y')^{-1}X'+Y'=H$. 
\end{itemize}

The corresponding $(*)$-part of our metric matrix (\ref{last.met.mat}) is exactly the 
lower right part of $Y_i$ which is diverging due to the divergence of $d_{r+j}(V_i)$ 
($i\rightarrow +\infty$) for any $j>0$. 
More precisely that $(g-r)\times (g-r)$ 
part is positive definite with all eigenvalues strictly diverge to $+\infty$. 

The diverging part (($g+r+j$)-th columns for $1\le j\le (g-r)$) yields $\mathbb{R}^{g-r}$ 
and the rest of part converges to the $2r$-dimension real torus with the metric matrix as 
(\ref{lim.Met.Mat}) below. 

\begin{Prop}\label{AV.v.fix}
In the above setting, the pointed Gromov-Hausdorff limit of our $V_i$ with fixed volume $1$ 
is isometric to 
$$(\mathbb{R}^{2r}/\mathbb{Z}^{2r})\times \mathbb{R}^{g-r}$$ 

\noindent
where the corresponding metric matrix of the first factor is 
\begin{equation}\label{lim.Met.Mat}
\begin{pmatrix}
\text{\large{$F$}}                  & \text{\large{$G$}} \\ 
\text{\large{${}^{t}G$}}   & \text{\large{$H$}} \\ 
\end{pmatrix}. 
\end{equation}
\end{Prop}

The proof follows straightforward from the discussion before the statement. 
Note that the metric matrix above (\ref{lim.Met.Mat}) 
corresponds exactly to the limit of 
$[V_i]_{i=1,2,\cdots}\in A_g$ inside the 
Satake-Baily-Borel compactification (cf., e.g., \cite[4.4]{Chai}). 
In conclusion, we have proved that 
\begin{Cor}
The Satake-Baily-Borel compactification 
$\bar{A_{g}}^{\text{\begin{CJK}{UTF8}{min}{\scalebox{.7}{\mbox{さ}}}\end{CJK}}BB}$ 
parametrizes the set of pointed Gromov-Hausdorff limits 
of $g$-dimensional 
principally polarized abelian varieties with fixed \textit{volumes}. 
\end{Cor}
This means that the Satake-Baily-Borel compactification \cite{Sat} can be differential 
geometrically naturally reconstructed, i.e., in the spirit of 
Gromov-Hausdorff. In our sequel with Y.Oshima \cite{OO}, 
we further identify $\bar{A_{g}}^{T}$ with another Satake's compactification. 



\section{Along holomorphic disks}\label{hol.fam}

In this section, we study the Gromov-Hausdorff limit 
along an arbitrarily taken meromorphic family which means 
(in this \S3 of our paper) a flat projective 
family $\pi^{*}\colon (\mathcal{X}^*,\mathcal{L}^*)\to \Delta^*$ 
where $\Delta^{*}:=\{t\in \mathbb{C}\mid 0<|t|<1\}\subset 
\Delta:=\{t\in \mathbb{C}\mid |t|<1\}$ which extends to 
some projective flat polarized family over whole $\Delta$. 
More precisely, fixing such $\pi^*$, we take a sequence of points $t(i)$ for $i=1,2,\cdots$ 
in $\Delta^{*}$ converging to the point $0\in \Delta$ and 
consider the Gromov-Hausdorff limit of corresponding metric spaces 
$\mathcal{X}_{t(i)}$ 
for $i=1,2,\cdots$. Of course, it could \textit{a priori} depends on the 
sequence $t(i)$ we take, 
but as a result of the following analysis, it turned out to be not! 
Note that in \cite{Od.Mg} and our \S \ref{Ag.TGC}, 
we considered all sequential Gromov-Hausdorff limits and hence our task here is to show such independence of the Gromov-Hausdorff limits along a fixed 
family as $\pi$ above and specify the subset consists of such limits. 

\subsection{Abelian varieties case}\label{Alg.deg.AV}

In this section, we remain on the 
principally polarized abelian varieties case. 

\begin{ntt}\label{AV.deg.setting}

This section focuses on the following situation. 
Take an arbitrary flat projective family of $g$-dimensional 
principally polarized abelian varieties over $\Delta^{*}$ 
which extends to some quasi-projective 
family over $\Delta$. 
Passing to a finite base change, 
we can and do assume that it admits (zero-)section, 
 i.e., is a family as algebraic groups and furthermore that 
 we have semi-abelian reduction over $0\in\Delta$ by the 
 Grothendieck semiabelian reduction. 
 
We write 
$(\mathcal{X}^{*},\mathcal{L}^{*})\to 
\Delta^{*}=\Delta\setminus \{0\}$ for such punctured family and the extension 
as $(\mathcal{X},\mathcal{L})\to \Delta$. 
Set the completion of the local ring of 
holomorphic functions at $0$ as $R:=\mathbb{C}[[t]]^{conv}$ 
(the convergent series local ring) and its fraction field 
$K:=\mathbb{C}((t))^{mero}$ 
(the field of meromorphic functions germs at 
$t=0\in \mathbb{C}$). 

From such germ at $0$ of this polarized family, 
one extracts the following data (``$DD_{ample}$'') 
as known to \cite{FC90} 
(which also at least partially go back to Mumford, Ueno, Nakamura, Namikawa etc). 
See \cite{FC90} for the details. 

\begin{enumerate}

\item \label{Raynaud.extension} 
The Raynaud extension $1\to T\to \tilde{\mathcal{X}}\xrightarrow{\pi} A\to 0$ over 
$R$

\item Ample line bundle $\tilde{\mathcal{M}}$ on $A$ and 
$\tilde{\mathcal{L}}:=\pi^{*}\mathcal{M}$, 

\item $X:={\it Hom}(T,\mathbb{G}_{m})$, $Y:={\it Hom}(\mathbb{G}_{m},T)$, the 
polarization morphism $\phi\colon Y\to X$, 
which is isomorphic in our case. 


\item $Y$-action on $\tilde{\mathcal{X}}$ as 
follows - there is a group homomorphism 
$\iota\colon Y \to \tilde{\mathcal{X}}
(K)$, given by $\{b(y,\chi)\in \mathcal{O}_{\mathcal{A}}\}$ 
via a (non-unique) isomorphism $\mathcal{X}\cong 
{\it Spec}(\oplus_{\chi}\mathcal{O}_{\mathcal{A}})$. 


\item Set $B'(y,\chi):={\it val}(b(y,\chi))$ 
and $B(y_{1},y_{2}):=B'(y_{1},\phi(y_{2}))$. $B$ is known to be a 
symmetric positive definite quadric form. 

\end{enumerate}

\end{ntt}

Then we analyze the asymptotic behaviour of the metrics 
along this degeneration as follows. 

\begin{Thm}\label{Ag.GH}

For $(\mathcal{X},\mathcal{L})\to \Delta$ as above, 
we suppose the extra assumption that Raynaud extension is the trivial extension 
(it is satisfied e.g. for the maximally degeneration case). 
Consider any sequence $t_{i} (i=1,2,\cdots) \in \Delta^{*}$ converging to $0$ 
then the fiber $\mathcal{X}_{t(i)}$ with rescaled flat K\"ahler metric (of diameter $1$) $\frac{d_{KE}(\mathcal{X}_{t(i)})}{{\it diam}(\mathcal{X}_{t(i)})}$ 
collapses to a $r$-dimensional real torus 
where $r$ is the torus rank of $\mathcal{X}_{0}$ with metric matrix 
$(cB(e_{i},e_{j}))_{i,j}$ for a basis $\{e_{i}\}$, $c\in \mathbb{R}_{>0}$ 
($c$ is for the rescaling to make the diameter $1$). 

\end{Thm}

Recall that as we explained at Notation \ref{AV.deg.setting}, 
any one parameter family of principally polarized abelian varieties  
can be reduced to the above form simply by the taking relative Picard 
space and then some finite base change. 
Hence the above result in particular 
confirms a conjecture by Kontsevich-Soibelman 
\cite[\S 5.1, Conjecture 1]{KS} for the abelian varieties case, 
and also can be regarded as abelian varieties variant as the conjecture of 
Gross-Wilson's 
\cite[Conjecture 6.2]{GW} or \cite[Conjecture 5.4]{Gross}. 
The author heard A.Todorov also had similar conjecture. 

\begin{proof}

By the triviality of the Raynaud extension, 
$\mathcal{X}$ is the fiber product over $R$ of a smooth projective family of 
$g$-dim principally polarized abelian varieties 
and another (degenerating) family of principally polarized 
abelian varieties 
which has maximal degeneration at $0\in Delta$. 
Then we apply simple \cite[Proposition 3.4]{Od.Mg} and 
we can easily reduce to the 
maximally degenerating case i.e. we can assume $\mathcal{X}_0$ is an algebraic torus, 
without loss of generality. 

In the maximally degenerating case, 
if we take a uniformizer $t$ of $0\in \Delta$ and take an isomorphism 
$T\cong \mathbb{G}_m^{r}$ which corresponds to a basis of $Y$,  
$y_1,\cdots, y_r$. 

From the standard way (definition) of the set of data we obtained at 
Notation \ref{AV.deg.setting}, 
the family $\mathcal{X}_{t}$ with $|t|\ll 1$ is well-known to be written as 
$\mathcal{X}_t=\mathbb{C}^{r}/M\cdot \mathbb{Z}^{2r}$ where

$$
M=\left(
\begin{array}{ccc|ccc}
2\pi i         &               &              &                &           &                     \\
                  & \ddots &                &             &    \huge{{\it log}(p_{i,j}(t))}&    \\ 
                   &            & 2\pi i     &               &          &                    \\
\end{array}
\right). 
$$

\noindent
where $p_{i,j}(t)$ is a symmetric matrix with coefficients in the 
meromorphic functions field $\mathbb{C}((t))^{\it mero}\subset \mathbb{C}((t))$. 
Indeed, $p_{i,j}(t)=\frac{1}{b(y_i,\phi(y_j))(t)}$. Recall $B(y_i,y_j)$'s definition from 
the notation, and that it is classically known to be a positive definite matrix. 
Taking the branch of ${\it log}(p_{i,j}(t))$ to make its absolute value of the 
imaginary part at most $2\pi$, the Siegel reduction is automatically done. 

From our arguments in \cite[proofs of 3.1, 3.3]{Od.Mg}, we know that 
the Gromov-Hausdorff limit of above is determined by the asymptotics of 
``$Y$"-(``imaginary") part of ${\it log}(p_{i,j}(t))$ for $t\to 0$ i.e., 
the orders of $p_{i,j}$. 
Hence, $\mathcal{X}_t  (t\neq 0)$ converges to $\mathbb{R}^{r}/\mathbb{Z}^{r}$ with 
metric matrix $B(-,-)$, appropriately rescaled to make the diameter $1$.
\end{proof}

Whether the following interesting phenomenon holds was asked, by A.Macpherson 
to whom we appreciate. 

\begin{Cor}[``valuative criterion of properness'']\label{val.cri.av}
Under the assumption of triviality of the Raynaud extension, 
the Gromov-Hausdorff limit of 
degenerating abelian varieties 
$\mathcal{X}_{t(i)}$ with rescaled flat 
K\"ahler metric (of diameter $1$) $\frac{d_{KE}(\mathcal{X}_{t(i)})}
{{\it diam}(\mathcal{X}_{t(i)})}$ 
does not depend on the converging sequences $t(i)\to 0 (i\to \infty)$. 
\footnote{In other words, the map from $\Delta^{*}$ sending $t$ to 
the underlying metric space of 
$\mathcal{X}_{t}$ with the rescaled K\"ahler-Einstein metric 
extend to $\Delta \to \{\text{compact metric spaces}\}$ as a continuous 
map in the sense of Gromov-Hausdorff. }
\end{Cor}

Actually, these \ref{Ag.GH} and \ref{val.cri.av} unconditionally holds for general degeneration of principally 
degeneration abelian varieties i.e., without the triviality assumption of 
the Raynaud extension. This will be proved as a part of joint work with 
Y.Oshima in a forthcoming paper \cite{OO}. 

\begin{Thm}[with Y.Oshima {\cite{OO}}]\label{Ag.GH.gen}
Let  $(\mathcal{X},\mathcal{L})\twoheadrightarrow C$ 
be as Notation \ref{AV.deg.setting}. (We do not assume triviality of the 
Raynaud extension). 
Consider any sequence $\{t(i)\}_{i=1,2,\cdots}\in \Delta^{*}$ converging to $0$ 
then the fiber $\mathcal{X}_{t(i)}$ with rescaled flat K\"ahler metric (of diameter $1$) $\frac{d_{KE}(\mathcal{X}_{t(i)})}{{\it diam}(\mathcal{X}_{t(i)})}$ 
collapses to a $r$-dimensional real torus 
where $r$ is the torus rank of $\mathcal{X}_{0}$ with metric matrix $B$ 
appropriately rescaled (to make the diameter $1$). 

In particular, the Gromov-Hausdorff limit of 
degenerating abelian varieties 
$\mathcal{X}_{t(i)}$ with rescaled flat 
K\"ahler metric (of diameter $1$) $\frac{d_{KE}(\mathcal{X}_{t(i)})}{{\it diam}(\mathcal{X}_{t(i)})}$ 
does not depend on the converging sequences $t(i)\to 0 (i\to \infty)$. 

\end{Thm}













\subsection{Algebraic curves case}
\label{Alg.deg.curves}

In this subsection, we analogously study 
asymptotics of the rescaled K\"ahler-Einstein metrics of bounded diameters 
along punctured meromorphic families of compact Riemann surfaces. 
We do \textit{not} logically require here 
the detailed construction of $M_{g}$ in \cite{Od.Mg}, which 
is described by the language of the 
Teichmuller space, its Fenchel-Nielsen coordinates 
and the pants decompositions. Instead, the following brief review 
of the statement provides enough context for our purpose here. 

The original analogue of Theorem \ref{AV.lim.gen} for compact Riemann surfaces 
case in \cite{Od.Mg} was as follows. 

\begin{Thm}[{\cite[Theorem 2.4]{Od.Mg}}]\label{GH.curves}
Let $\{R(i)\}_{i=1,2,\cdots}$ be an arbitrary sequence of compact Riemann surfaces of fixed genus $g\geq 2$. 
Suppose $(R(i), \frac{d_{{\it KE}}(R(i))}
{\it diam(R(i))})$ $(i=1,2,\cdots)$ 
converges in the Gromov-Hausdorff sense. 
Here $d_{KE}$ denotes the K\"ahler-Einstein metric on each $R(i)$ and its diameter is ${\it diam}(R(i))$. 

Then the Gromov-Hausdorff limit is either 
\begin{enumerate}
\item \label{conv.to.graph} a metrized (finite) graph of diameter $1$ or 
\item \label{conv.to.surf} a compact Riemann surface of the same genus. 
\end{enumerate}

Since the Deligne-Mumford compactification $\bar{M_g}^{{\it DM}}$ with 
the complex analytic topology 
is compact, by passing to a subsequence if necessary, 
we can assume that $[R(i)]_{i=1,2,\cdots}$ converges to some  $R(\infty)^{DM}
\in \bar{M_g}^{{\it DM}}$ without loss of generality. 
Then, the case (\ref{conv.to.graph}) happens if and only if 
$R(\infty)^{DM}$ is non-smooth 
stable curve and in that case, 
the combinatorial type of the graph is a contraction of the dual graph of the 
corresponding stable curve $R(\infty)^{DM}$ 
i.e., the limit of $[R(i)]_{i=1,2,\cdots}$ 
in the Deligne-Mumford compactification of the moduli of curves 
$\bar{M_{g}}^{{\it DM}}$, 
 with 
non-negative metrics (possibly zero) on each edges. 

Conversely, any metrized dual graph of the stable curve of genus $g$ with diameter $1$ can occur as the Gromov-Hausdorff limit in case $(i)$. 
\end{Thm}

\begin{Cor}[{cf., \cite[\S 2.3, \S 3.2]{Od.Mg}}]

$$\bar{M_{g}}^{T}:=M_{g}\sqcup S_{g}^{wt}$$

\noindent
with a certain natural topology 
is a compactification
\footnote{i.e., a compact Hausdorff topological space which 
contains $M_{g}$ as an open dense subset.}
 of $M_{g}$ with complex analytic topology, 
where $S_{g}^{wt}$ denotes the moduli space of metrized finite graphs, 
with weights $w(v_{i})$ on each vertex $v_{i}$, whose 
underlying topological spaces 
satisfy purely combinatorial condition: 
$$v_1(\Gamma)+b_{1}(\Gamma)+\sum_{i}w(v_{i})=g.$$ 
Here, 
we denote the number of $1$-valent vertices as 
$v_1(\Gamma)$ and denote the first betti number of $\Gamma$ as $b_{1}(\Gamma)$. 
The above condition is nothing but the characterization of 
finite graphs which can appear as the dual graph of some Deligne-Mumford 
stable curves of genus $g(\ge 2)$ and the weights encode the genera of 
the components of the normalization. 

\end{Cor}

In the following arguments, we specify which 
metrized graphs can appear as the Gromov-Hausdorff limits along 
\textit{meromorphic} punctured family while also proving that such limits are 
well-defined. We start with setting up the Kuranishi space of 
stable curves. 

\subsubsection{Semi-universal deformations}\label{Wolp}

Basic deformation theory of stable curve $R$ tells us that we have a 
semiuniversal (un-obstructed) deformation. 
Its tangent space $Ext^{1}(\Omega^{1}_{R},\mathcal{O}_{R})$ maps surjectivly to 
local deformation tangent space ${\it Def}^{{\it loc}}\cong\mathbb{C}^{m}$ 
whose $i$-th coordinate corresponds to smoothing one 
of $m$ nodes $x_{i}\in R$. 
We first discuss at the semi-universal deformation level in this subsection 
and then apply (restrict) that to one parameter deformations later at 
\S \ref{1par.curve}. 

We anyhow need 
the Wolpert's fundamental results in \cite{Wol} (cf., e.g. also \cite{OW08}) on  asymptotics of the hyperbolic metrics of compact Riemann surfaces 
along an arbitrary  
degeneration to a stable curve $R$. His constructions of smoothing and approximation 
of the hyperbolic metric are explained as Step \ref{plumbing}, 
Step \ref{grafting} below respectively. 
We reproduce his results for the convenience of readers 
and to set up the stage of our later discussions. 

\begin{Step}[``plumbing surfaces'']\label{plumbing}

Recall that there is a semiuniversal algebraic deformation $\mathcal{U}
\twoheadrightarrow Z$ on an \'etale cover (variety) $Z$ of ${\it Def}(R)={\it Ext}^{1}
(\Omega^{1}_{R},\mathcal{O}_{R})$. We re-construct its analytic germ in a differential geometric way as follows. We first take a equi-singular 
deformation which is a restriction of $\mathcal{U}\to Z$ to a closed subset $Z'$of $Z$. This can be also constructed as the product of 
universal deformation of each components (with nodes marked). We denote this as $\{R_{s}\}_{s\in Z'}$. 

We take the normalizations of $R_{s}$s which of course form a family $R_{s}^{\nu}$ again. 
Then around the (section formed by) 
preimages of $i$-th node(s) ($1\le i\le m$) $x_{i}(s)$ in $R_{s}^{\nu}$ which we  
denote 
by $p_{i}(s)$ and $q_{i}(s)$, we take a (holomorphic family of) local coordinates $z_{i}(s), w_{i}(s)$ around $p_{i}(s)$ and $q_{i}(s)$ respetively 
so that 
$$z_{i}(s)(p_{i}(s))=0,$$ 
$$w_{i}(s)(q_{i}(s))=0.$$ 

Fix a small enough positive real number $c_{*}<1$. 
Then we construct a small deformation of $R_{s}$ as 

$$R_{s,\vec{t}}:=\biggl( R_{s}\setminus 
\biggl(\bigsqcup_{i}\biggl\{|z_{i}(s)|<\frac{|t_{i}|}{c_{*}}\biggr\}
\sqcup \biggl\{|w_{i}(s)|<\frac{|t_{i}|}{c_{*}}\biggr\}\biggr)\biggr)/\sim,$$ 

\noindent 
where for $\vec{t}=\{t_{i}\}_{i}\in \mathbb{C}^{m}$ 
with $|t_{i}|<c_{*}^{4}$ and the equivalence relation $\sim$ is defined on the  disjoint union of pairs of sub-annuli  
$$\bigsqcup_{i} 
\biggl(
\biggl\{\frac{|t_{i}|}{c_{*}}\le |z_{i}(s)|\le c_{*}\biggr\}\sqcup 
\biggl\{\frac{|t_{i}|}{c_{*}}\le |w_{i}(s)|\le c_{*} \biggr\}
\biggr)
$$ as 
$$z_{i}(s)\sim w_{i}(s) \iff z_{i}(s)w_{i}(s)=t_{i}.$$ 
Clearly $R_{s,\vec{t}}$ form a holomorphic flat family of compact Riemann surfaces  (equisingular along $\vec{t}=0$). 
We call the image of the annuli in the plumbed Riemann surface $R_{s,\vec{t}}$ 
as \textit{collar}s following \cite{Wol} or sometimes ``neck''s in literatures. 
In this way, we get 
a family $R_{s,\vec{t}}$ where $\vec{t}=(t_{1},\cdots,t_{m})\in \mathbb{C}^{m}$ with $|t_{i}|\ll 1$, 
and hence an analytic slice transversal to the equisingular locus $Z'$ inside the semi-universal deformation space $Z$. 
Note that 
the construction \textit{does} depend on the local coordinates 
$z_{i}(s),w_{i}(s)$.

\end{Step}

\begin{Step}[``grafting metric'']\label{grafting}

Next, we set a negative real number $a_{0}<0$ with $|a_{0}|\ll 1$ (depending on the construction of Step1) fixed and a 
$C^{\infty}-$(``bump'') function $\eta\colon \mathbb{R}\to \mathbb{R}$ so that 

$$\eta(a)= \begin{cases} 1 & \text{ if } a\le a_{0} \\
    \in [0,1] & \text{ if } a_{0}<a<0 \\ 
    =0 & \text{ if } a>0. \\ 
  \end{cases}
$$

We call the annuli 
$\{\frac{|t_{i}|}{e^{a_{0}}c_{*}}<|z_{i}(s)|<e^{a_{0}}c_{*}\}$
the \textit{collar core}s and the complement in the collars i.e., the 
set of pairs of annuli $B_{z_{i}}:=\{e^{a_{0}}c_{*}\le |z_{i}(s)|\le c_{*}\}$, $B_{w_{i}}:=\{e^{a_{0}}c_{*}\le |w_{i}(s)|\le c_{*}\}$ 
the \textit{collar band}s. 
Then we set a function $\eta_{z_{i}}$ (resp., $\eta_{w_{i}}$) 
at an open neighborhood of $B_{z_{i}}$ (resp., $B_{w_{i}}$) as 
$$
\eta_{z_{i}}:=\eta\biggl({\it log}\frac{|z_{i}(s)|}{c^{*}}\biggr)  (\text{resp., }  
\eta_{w_{i}}:=\eta\biggl({\it log}\frac{|w_{i}(s)|}{c^{*}}\biggr)). 
$$

Now we ``glue'' the complete hyperbolic metric on $R_{s}\setminus \{x_{1},\cdots,x_{m}\}$ (i.e. smooth locus) which we denote as 
$dg^{2}_{s}$ and the local model metric (``with long neck'') $dg^{2}_{loc,t_{i}}$ 
around $x_{i}$, defined as below for $|t_{i}|\ll 1$.  
The gluing  uses the above bump functions $\eta_{z_{i}}$ and $\eta_{w_{i}}$. 
Here the local model metric $dg^{2}_{{\it loc},t_{i}}$ 
is defined as the restriction of 
$$\biggl(\dfrac{\pi}{{\it log}|t_{i}|}{\it csc}\dfrac{\pi{\it log}|z_{i}(s)|}{{\it log}|t_{i}|}\biggl|\dfrac{dz_{i}(s)}{z_{i}}\biggr|\biggr)^{2}.$$ 
In particular, it does not essentially depend on $s$.

Then the actual definition of the smooth (hermitian) 
metrics family $dg^{2}_{s,\vec{t}}$ on $R_{s,\vec{t}}$ by Wolpert, which he calls \textit{grafting}, is as follows. 

\begin{equation}\label{grafted.metric}
dg_{s,\vec{t}}^{2}:= \begin{cases} (dg^{2}_{s})^{1-\eta_{z_{i}}}\cdot (dg^{2}_{{\it loc},t_{i}})^{\eta_{z_{i}}} &\text{ around }B_{z_{i}}\\ 
dg^{2}_{{\it loc},t_{i}} & \text{ at the collar core of } x_{i}\\
(dg^{2}_{s})^{1-\eta_{w_{i}}}\cdot (dg^{2}_{loc,t_{i}})^{\eta_{w_{i}}} & \text{ around }
B_{w_{i}} \\
dg_{s}^{2} & \text{ otherwise }\\
\end{cases}\\ 
\end{equation}

\noindent 
The above definition by patching is well-defined since at the collar core we have 
$\eta_{z_{i}}=\eta_{w_{i}}=1. $

Again, we do the above procedure (Step \ref{grafting}) 
of grafting the metrics for all nodes $x_{i} (i=1,\cdots,m)$ 
simultaneously. 
See \cite[\S 3]{Wol}, \cite[\S 2]{OW08} for more details if needed. 
As a result of the above two steps construction, we get a smoothing family of $R_{s}$, over a $(t_{1},\cdots,t_{m})$-polydisc 
on $\mathbb{C}^{m}$ which we denote as $R_{s,\vec{t}}$ 
and $C^{\infty}-$hermitian metrics family $dg_{s,\vec{t}}^{2}$ on $R_{s,\vec{t}}$. 

\end{Step}

\begin{Step}
The crucial result of \cite{Wol} we use 
compares the \textit{grafted} metrics $dg_{s,\vec{t}}^{2}$ with the hyperbolic  metrics on $R_{s,\vec{t}}$. In conclusion, he proved that they have the ``same asymptotic  behaviour''. 

From the construction of Step 1 and 2, it is obvious that the grafted metric $dg_{s,\vec{t}}^{2}$ is 
the same as local model hyperbolic metric $ds^{2}_{loc,t_{i}}$ on the 
collar core with respect to $x_{i}$ i.e., 
$$\biggl\{\frac{|t_{i}|}{e^{a_{0}}c_{*}}<|z_{i}(s)|<e^{a_{0}}c_{*}\biggr\}$$ 
and coincides with the restriction of the original hyperbolic metric $ds^{2}_{s}$ on 
$R(s)$. On the collar bands i.e., 
$$\{e^{a_{0}}c_{*}\le |z_{i}(s)|\le c_{*}\}\sqcup \{e^{a_{0}}c_{*}\le 
|w_{i}(s)|\le c_{*}\},$$
the grafted metric is a mixture of $dg_{s}^{2}$ and $dg^{2}_{loc,t_{i}}$. 
The crucial result we will use is the following. 

\end{Step}

\begin{Fac}[{\cite[Lemma3.5, \S 3.4, 4.2]{Wol} cf., also \cite[p690]{OW08}}]\label{gr.hyp}
The grafted metric $dg_{s,t}^{2}$ is asymptotically equivalent to hyperbolic metric $dg_{hyp,s,\vec{t}}^{2}$ in the sense that 
$$dg_{s,\vec{t}}^{2}=dg_{hyp,s,\vec{t}}^{2}\cdot (1+O(\sum_{i}(log|t_{i}|)^{-2}),$$ 
\noindent 
when $\vec{t}\to 0.$ 
\end{Fac}

From the above gluing construction of the metrics by Wolpert, we 
straightforwardly see that 
\begin{eqnarray*}
{\it diam}(i\text{-th collar core of } (R_{s,\vec{t}},dg_{s,\vec{t}}^{2}))
=\\ 
\int_{z=\frac{|t_{i}|}{c^{*}}}^{c^{*}}\frac{\pi}{\log|t_{i}|}\cdot 
{\it csc}\biggl(\frac{\pi \log|z_i|}{\log|t_{i}|}\biggr)\frac{dz_i}{z_i}+O(1),
\end{eqnarray*}

\noindent
for $t_i\to 0$, where the last $O(1)$ contribution comes from 
the ${\it arg}(z(i)) (\in \mathbb{R}/2\pi\mathbb{Z})$-direction 
of the collars and the existence of the collarbands. 
By putting $|t_{i}|=e^{T_{i}}$, $z=e^{T_{i}x_{i}}$ with $T_{i},x_{i}\in \mathbb{R}$, 
the above can be re-expressed as 

$$
\int_{T_{i}x_{i}=\log(|t_{i}|)}^{T_{i}x_{i}=\log(c^{*})}\frac{\pi}{T_{i}}{\it csc}(\pi x_{i})d(T_{i}x_{i})+O(1). 
$$

Then a simple calculation shows

\begin{eqnarray*}
\int_{T_{i}x_{i}=\log(|t_{i}|)}^{T_{i}x_{i}=\log(c^{*})}\frac{\pi}{T_{i}}{\it csc}(\pi x_{i})d(T_{i}x_{i})
& & =\int_{T_{i}x_{i}=\log(|t_{i}|)}^{T_{i}x_{i}=\log(c^{*})}{\it csc}(\pi x_{i})dx_{i} \\
& &:=\int_{T_{i}x_{i}=\log(|t_{i}|)}^{T_{i}x_{i}=\log(c^{*})}
\frac{1}{{\it sin}(\pi x_{i})}dx_{i} \\
& & =\int_{T_{i}x_{i}=\log(|t_{i}|)}^{T_{i}x_{i}=\log(c^{*})}
\frac{dx_{i}}{\pi x_{i}}+O(1)
\\ 
& & =2\log(-\log(t_{i}))+O(1), \\
\end{eqnarray*}

\noindent
for $t_{i}\to 0$. 
The above calculation is reflecting that, for fixed $i$, the metrics 
$dg^{2}_{loc,t_{i}}$ uniformly converge to $\biggl(\dfrac{dz_{i}}{|z_{i}|{\log}|z_{i}|}
\biggr)^{2}$ on any compact subsets of $\{0<|z(s)|<c_{*}\}$ when $|t_{i}|\to 0$. 
This gives the proof of the Lemma \ref{gr.GH}, when combined with \ref{mod.map}. 

If we are allowed to use our slight extension of Morgan-Shalen-Boucksom-Jonsson type 
compacfication for more general \textit{gluing function} (Appendix \ref{MS.gen.fun}) 
plus the result of Abramovich-Caporaso-Payne \cite{ACP}, 
we can rephrase and summarize the above outcome into the following statement 
\ref{mg.hyb.tgc}. 
We refer the readers 
to Appendix \ref{MS.gen.fun} and \cite{ACP} for the details 
of such preparation and we simply use it without recalling it. 
However, if one 
would not have the background and not much interested, then one could possibly 
skip the following as it essentially just rephrases the above discussions. 
For those interested, please read Appendix (especially \S \ref{MS.gen.fun}) and 
\cite{ACP} as a preparation. Let me only roughly mention that 
the Morgan-Shalen-Boucksom-Jonsson partial compactification and our extension 
attach ``dual intersection complex'' to a given space. 

\begin{Thm}\label{mg.hyb.tgc}
There is a natural homeomorphism between 
our compactification $\bar{M_{g}}^{T}$ (in \cite{Od.Mg}) and 
a variant of Morgan-Shalen-Boucksom-Jonsson compactification 
of $M_g$ introduced at our \S \ref{MS.gen.fun}. That is, 
$$\bar{M_{g}}^{T}\cong 
\bar{M_{g}}^{\it hyb}_{,(\log(-\log|\cdot|))},$$ 
\noindent
in the notation of \S \ref{MS.gen.fun} in our appendix. 
More precisely, the above homeomorphism extends the 
identity of $M_{g}$ and 
preserves the corresponding weighted metrized graphs or compact 
Riemann surfaces' isomorphism classes when we see the boundary 
$\partial \bar{M_{g}}^{\it hyb}_{,(\log(-\log|\cdot|))}$ as the moduli of such 
metrized graphs by \cite{ACP}. 
\end{Thm}

\begin{proof}[proof of Theorem \ref{mg.hyb.tgc}]
We use the above analysis of (approximation of) 
hyperbolic metrics based on \cite{Wol}, for the proof. 
From our construction of \S \ref{MS.gen.fun}, the boundary of the right hand side 
compactification is the dual complex of the boundary of the Deligne-Mumford 
compactification stack $\partial \bar{\mathcal{M}_g}^{DM}$, i.e., 
the quotient of the dual intersection complex of the 
boundary divisors in charts, divided by the natural equivalence relation 
induced by the stack structure. We denote such topological space (dual complex) 
as $\Delta_g$. On the other hand, 
Abramovich-Caporaso-Paybe \cite{ACP} 
identified $\Delta_g$ with $S_g^{wt}$ in our notation, preserving the 
real affine structures on the both sides, 
thus we have a canonical bijection between the above two spaces 
extending the identity on $M_g$. 
We refer to \cite{ACP} for more details, if needed. 

To show that the two topologies, the (generalized) hybrid topology (\cite{BJ}, 
\ref{MS.gen.fun}) and the (weighted) Gromov-Hausdorff topology (cf., \cite{Od.Mg} and our subsection 1.1)  coincide, it is enough to see that for any given net (in the sense of Bourbaki)
$\{x_i\in (M_g \sqcup \Delta_g) \}_{i\in I}$ converging in both topologies converge to 
the same point. Since both topologies extend the usual analytic topologies on $M_g$ and also the natural euclidean topology on the boundary $\Delta_g$, we can further suppose that all $x_i$ are in $M_g$ and it converges to $y \in \Delta_g$ for the hybrid topology (resp., $z\in \Delta_g$ for the (weighted) Gromov-Hausdorff topology). What we want to show is $y=z$.

Passing to its sub-net if necessary, we can suppose that $\{x_i\}$ converges to a stable curve $R$ in the Deligne-Mumford compactification $\bar{M_g}$ with analytic topology. Then we can lift the net, again by passing to subnet if necessary, to the level of ${\it Def}^{\it hyb}(R)$ (recall our notation from previous section) which we denote by
$\{\tilde{x_i}\}_{i\in I}$. Then we want to see the equivalence of convergence of $\{\tilde{x_i}\}_{i\in I}$ to some lift $\tilde{y}$ in $\Delta^{\it loc}(R)$ in the hybrid topology and in the (weighted) Gromov-Hausdorff topology. On the other hand, it actually straightforward follows from our proof of Proposition \ref{GH.lim} since the convergence in the hybrid topology is the convergence of ratio of logarithms of the absolute values of the local coordinates corresponding to the nodes, while it is the same for the weighted Gromov-Hausdorff topology as we analyzed by using \cite{Wol}. We complete the proof.

\end{proof}

\begin{Rem}
Theorem \ref{mg.hyb.tgc} somewhat refines the slogan after 
Abramovich-Caporaso-Payne \cite{ACP} that ``the moduli of skeleton is 
skeleton of the moduli''.\footnote{as addressed in A.Macpherson's talk}

For the case of $A_g$, in our joint work with Y.Oshima \cite{OO}, 
we also proved that $\bar{A_{g}}^{T}$ is identical to  
Morgan-Shalen-Boucksom-Jonsson compactification of 
toroidal compactifications of $A_{g}$ in a slightly extended sense 
(\S \ref{toroidal.compactification.MS}). By definition, it in particular 
proved that the dual complex (in the sense of \cite[\S6]{ACP}, Appendix) 
of toroidal compactification of $A_{g}$ 
does not depend on the choice of cone decompositions and is 
canonically homeomorphic to our tropical moduli space 
$T_{g}=\partial \bar{A_{g}}^{T}$. This is the 
abelian variety version of the Abramovich-Caporaso-Payne theorem 
\cite[Theorem 1.2.1]{ACP}. 
 \end{Rem}






\subsubsection{One parameter families of curves}\label{1par.curve}


Now we discuss one parameter family of stable curves. 
Let us set the notation first. 

\begin{ntt}\label{ntt.curve}

Let $\mathcal{X}\to \Delta$ be a flat projective family of curves over the disk 
$\Delta=\{|t|<1\}$, whose central fiber $\mathcal{X}_{0}=:R$ is 
a stable curve. We denote its restriction over the punctured disk $\Delta^*$, the 
smooth projective family as $\mathcal{X}^{*}\to 
\Delta^{*}$. Suppose $\mathcal{X}_{0}$ has nodes $x_{1},\cdots,x_{m}$ and around $x_{i}\in \mathcal{X}$, we have local equation 
$zw=t^{m_{i}}$ with $A_{m_{i}-1}$-singularity. 
The following is also a well-known fact and easy to confirm. 

\begin{Fac}\label{mod.map}
Consider the natural morphism $f\colon \Delta
\to {\it Def}^{loc}$ associated to our family $\mathcal{X}\to \Delta$. 
Then ${\it ord}_{t}(f^{*}t_{i})=m_{i}$, where $t_{i}$ is a coordinate 
corresponding to the direction of smoothing $x_{i}$ out. 

\end{Fac}

\end{ntt}

\begin{Prop}\label{GH.lim}
Consider any sequence $t(i) (i=1,2,\cdots) \in \Delta^{*}$ converging to $0$ and 
suppose that the fiber $\mathcal{X}_{t(i)}$ with rescaled hyperbolic metric (of diameter $1$) $\frac{d_{KE}(\mathcal{X}_{t(i)})}{{\it diam}(\mathcal{X}_{t(i)})}$ 
collapses to a metrized finite graph (cf., \cite[2.4]{Od.Mg}). 
Then the metrized graph is nothing but 
the dual graph of $\mathcal{X}_{0}$ whose edges have all the same lengths. 

\end{Prop}

\begin{Rem}
We note that the above identification of Gromov-Hausdorff limits for 
\textit{meromorphic family} is \textit{not} sufficient for constructing 
the compactification itself. 
\end{Rem}

Benefiting from the above fact, our proof of \ref{GH.lim} is reduced to the following. 

\begin{Lem}\label{gr.GH} 
The Gromov-Hausdorff limit of the rescaled grafted metrics on 
$\mathcal{X}_{t(i)}$ with the diameter $1$ when $t(i)\to 0$ 
is nothing but 
the dual graph of $\mathcal{X}_{0}$ whose edges have all the same lengths. 
\end{Lem}

\begin{proof}[proof of Proposition \ref{GH.lim}]

From our discussion above, it is enough to analyze the 
diameter contribution of the collars in $R_{s,\vec{t}}$. 
Recall that the grafted metric is 
nothing but the local model (``long neck'') metric on the collar core while, 
on the collar bands,  the grafted metric lies between the hyperbolic metric on $R$ 
and the local model, thus its contribution to the diameter is bounded above. 
Then the assertion follows straightforward from our previous discussions. 
\end{proof}

The above claim especially shows that the Gromov-Hausdorff 
limit one can get as above in $\partial \bar{M_g}^{T}$ 
along holomorphic one parameter deformations consist of only \textit{finite} points! 
In particular, we also have an analogue of 
Corollary \ref{val.cri.av} for curve case as well. 

\begin{Cor}\label{val.cri.curve}
Once we fix the family $(\mathcal{X}^*,\mathcal{L}^*)\to \Delta^*$, 
the Gromov-Hausdorff limit of $\mathcal{X}_{t(i)}$ with rescaled hyperbolic metric (of diameter $1$) $\frac{d_{KE}(\mathcal{X}_{t(i)})}{{\it diam}(\mathcal{X}_{t(i)})}$ 
for converging sequences $t_{i}\to 0$ $(i\to \infty)$ do not 
depend on the choice of the sequence we take and 
the set of such limits  form only 
a finite subset inside $\partial \bar{M_g}^{T}$. 
\end{Cor}

\begin{Rem}
As we briefly introduced at our previous paper \cite{Od.Mg}, 
L.Lang \cite{LL} also introduced the following notion of convergence of 
compact Riemann surfaces to a metrized finite graph (which he calls tropical curve 
in the paper), which we analyze and compare with our compactification 
(cf., also his discussion at \cite[v2, \S 1.3]{LL}). 

\begin{Def}[{\cite[Definition 1.1]{LL}}]\label{LL.conv}
\footnote{The expression is somewhat different from the original but the equivalence with it is just a matter of unwinding his definition.}
A sequence of compact Riemann surfaces of fixed genus $g(\ge 2)$ 
$R_{i}|_{i=1,2,\cdots}$ converges in the 
sense of ``\textit{tropical convergence}''(\cite{LL}) to 
a metrized finite graph $C$ if the following holds: 
\begin{quote}
$R_{i}$ converges to a stable curve $R_{\infty}$ in $\bar{M_{g}}^{DM}$ 
while shrinking a set of simple geodesics $l_{a}(R_{i})$ and $C$ underlies a dual graph $\Gamma$ of $R_{\infty}$ such that 
$${\it length}(l_{a}(R_{i}))=c_{i}(1+o(1))\frac{1}
{{\it length}(l_{a}(\Gamma))}$$ for $i\to \infty$, for some constants $c_{a}$ which are independent of $i$. 
\end{quote}
Here, $l_{a}(\Gamma)$ means the edge of $\Gamma$ corresponding to the 
node comes from the shrinking of $l_{a}(R_{i})$ for $i\to \infty$. 
\end{Def}

It directly follows from the known fact \ref{gr.hyp} (cf., the whole \S
\ref{Wolp} of our review), 
that his notion of convergence \ref{LL.conv} 
is \textit{different} from our (weighted) Gromov-Hausdorff convergence 
and actually \textit{coincides} with 
the Morgan-Shalen-Boucksom-Jonsson compactification 
in the slightly extended sense for stacks in the sense of 
Appendix \ref{DMstack.MS}. 
We only sketch the proof as it is quite simple: 
the approximating grafted metric (\ref{grafted.metric}) 
has the simple closed geodesic of length proportional to $\frac{1}{|\log(t)|}$ and 
we can apply such approximation result \ref{gr.hyp} to a finite set of 
the semi-universal local deformations of stable curves, covering the 
whole (compact) boundary of the Deligne-Mumford compactification $\partial 
\bar{M_{g}}^{DM}$. 
\end{Rem}

\begin{Rem}
From the above results, we get a morphism 

$$R\colon M_g^{\it an}\to \bar{M_g}^{T},$$

\noindent
where $M_g^{\it an}$ denotes the Berkovich analytification of $M_g$ over 
the complete discrete valuation field $\mathbb{C}((t))$, 
simply by considering reduction (compare with \cite{ACP}). 
The map $R$ is neither continuous nor anti-continuous, in fact 
it is rather a combination of anti-continuous reduction map over the inner part $M_g$ 
and continuous tropicalization map over the boundary $\partial \bar{M_g}^{T}$. 
Again, from our results in \S \ref{AV.deg.setting}, 
we also have a completely analogous map 
$A_g^{\it an}\to \bar{A_g}^{T}$ which is neither continuous nor anti-continuous. 
\end{Rem}

\subsection{Torelli maps} 

It is natural to think how or whether the classical period map 
$$t^{alg}\colon M_{g}\hookrightarrow A_{g}$$ 
extends between our two compactifications $\bar{M_{g}}^{T}$ and $\bar{A_{g}}^{T}$. 
 
Let us briefly recall the recent study of Torelli problem in tropical setting by other 
mathematicians before. 
For a \textit{unweighted} (or weighted with zeroes) metrised graph $\Gamma$, 
the \textit{tropical Jacobian} \cite{MZ}, \cite{BMV} is simply 
$H_1(\Gamma, \mathbb{R}/\mathbb{Z})$ 
with the following positive definite quadratic form $Q$. It is defined as 
$$Q(\sum_{e: \text{edge}}\alpha_e \cdot e):= \sum_{e} \alpha_e^2\cdot l(e)$$ 
for each $1$-cycle $\sum_{e: \text{edge}}\alpha_e \cdot e$ 
where $l(-)$ denotes the length function. 
Later, this turned out to be equivalent as the skeleton of (generalized) Jacobian 
in non-archimedean sense by \cite{Viv12,BR}. 

By mapping any tropical curve to its tropical Jacobian, 
\cite{CaV} and \cite{BMV} essentially 
established the existence of a natural map 
$$
t^{Trop}\colon (S_{g}^{wt}\setminus S_{g}^{wt,tree})\to T_{g}
$$
which is not only continuous but also compatibile 
with their ``stacky fan" structure \cite{BMV} over the cones of these. 
Here, $S_{g}^{wt,\text{tree}}$ denotes 
the closed locus of $S_{g}^{wt}$ which parametrizes those which underly trees, 
that is disjoint from $S_{g}^{wt,o}$. 
(If $\Gamma$ is a tree, 
then the tropical Jacobian of \cite{MZ}, \cite{BMV} is just a point so that we cannot rescale 
to make the diameter $1$. ) 
The Torelli property i.e., the injectivity of the above 
does \textit{not} literally hold even in $g=2$ case as pointed out in \cite{MZ}. 
Indeed, the closure of only one of the $2$-cells of $S_{2}$ which parametrizes those 
without connecting edge maps onto $T_{2}$. 
Nevertheless, they proved that it is 
``generically one to one" \cite{CV}, \cite{BMV}. 

Now, it is natural to ask the following question of Namikawa-Mumford-Alexeev type (cf., \cite{Nam1}, \cite{Ale2}) 
i.e., about the extension of $t^{alg}\colon M_g\hookrightarrow A_{g}$ 
to compactifications. 
To be more precise, by combining the above two ``period maps" 
$$t^{alg}\colon M_{g}\to A_{g}$$ and 
$$t^{Trop}\colon (S_{g}^{wt}\setminus S_{g}^{wt,\text{tree}})
\rightarrow T_{g}$$ 
we get a map $$t_{g}(:=t^{alg}\sqcup t^{Trop})\colon 
M_{g}^{T}\setminus S_{g}^{wt,\text{tree}}\to A_{g}^{T}. 
$$
Now it is natural to ask the questions of continuity of the map $t_{g}$. 
It essentially asks the compactibility of Jacobians and tropical Jacobians. 
Although we have reviewed this theory of (tropical) Jacobians 
for the sake of expository 
completeness, we are afraid that the answer is no! 

\begin{Prop}\label{Torelli.Not}
The above map $t_{g}$ is \underline{not} continuous for any $g>1$. 
\end{Prop}

\begin{proof}
Although we see this failure of continuity in a more systematic way 
later, we give explicit examples with $g=2$ here. 
We use the fact that usual $\bar{t}\colon 
\bar{M}_{g}^{DM}\to \bar{A_{g}}^{Vor}$ 
of Mumford-Namikawa \cite[\S 18]{Nam1} is isomorphism, 
which seems to be well-known to experts (cf., e.g. \cite[Example 18.14]
{Nam1}). From more modern perspective, it can be re-explained a 
little more simpler as follows. First, the pairs of 
their degenerate abelian varieties and their theta divisors form semi-log-canonical 
pairs \cite[3.10]{Ale0}, \cite{Ten}.  Hence by adjunction, we conclude that 
such theta divisors are connected nodal curves with ample canonical classes, i.e., stable curves. 
(Note that, from our modular interpretation, the above isomorphism can be ascended to 
stacky level $\bar{\mathcal{M}_g}^{DM}\cong \bar{\mathcal{A}_g}^{Vor}$. )

Take a stable curve $C_{0}:=C_{1}\cup C_{2}$ which can be described as follows. 
Two irreducible components are isomorphic $C_{1}\cong C_{2}$ and 
are rational curves with one self-intersecting nodal singularity $p_{i} 
(i=1,2)$ each. 
Furthermore, $C_{1}$ and $C_{2}$ intersect transversally at one nodal point $q$. 
Take the semi-universal deformation of $C_{0}$, which is three dimensional 
smooth germ with normal crossing discriminant 
divisors $D=D_{1}\cup D_{2}\cup D'$ components of 
which are corresponding to local smoothing of $p_{i}$ and $q$ respectively. 
We take a complex analytic coordinates $(z_{1},z_{2},z_{3})$ corresponding 
to the divisor $D$, i.e. which satisfy 
$[z_{i}=0]=D_{i}$ for $i=1,2$ and $[z_{3}=0]=D$. 
Consider analytic family of curves $\{[C_{t}]=\varphi(t) \in M_{g}\}_{t}$ over 
the unit disk 
$\Delta$ described as $z_{i}=t^{a_{i}}$ with 
$a_{i}\in \mathbb{Z}_{>0}$.

Now, suppose that the map $t_g$ is continuous. 
Then the family $\varphi(t)$ in $M_2$ converging to a point in 
$\partial \bar{M_{g}}^{T}$ corresponding to 
the graph consists of two circles joined by an edge, which looks like a handcuff.  All the three edges have 
same lengths by Theorem \ref{mg.hyb.tgc} for $t\to 0$ (independent of 
a convergent sequence of $t$ we take). We denote a point corresponding to this 
metric graph by $h$. Its tropical Jacboian $t_{g}(h)$ is 
a $2$-dimensional real flat tori $S^{1}(c)\times S^{1}(c)$ 
appropriately rescaled by a positive constant $c$ 
with the diameter $1$. (The exact value of $c$ is $\sqrt 2$ after all). 
Here, the value inside the 
parathesis denote the length of the circumferences of the metrized circle $S^1$. 

On the other hand, by Theorem \ref{Ag.GH}, 
$t_g(p_i)$ converges in the Gromov-Hausdorff 
sense to $S^{1}(ca_{1})\times S^{1}(ca_{3})$ with appropriate positive 
constant $c$ to make the diameter $1$. Clearly that metric space depends on the 
parameters $a_{i}$ which contradict to the above. 
For any $g>2$, we can create a counterexample to the continuity of 
$m_{g}$ as the above counterexample of $g=2$ 
attached with $g-2$ elliptic tails. 
\end{proof}

\begin{Rem}
If we think of the fact that 
our compactification 
$\bar{M_{g}}^{{\it hyb}}$, 
to be introduced at our Appendix \ref{DMstack.MS} 
in much more general context, coincides with \cite{LL}'s compactification, 
then the existence of 
continuous map $\bar{M_{g}}^{{\it hyb}}\to \bar{A_{g}}^{T}$ also follows from 
our later disucssion, as a special case of \ref{functoriality.MS}. 
For further details, we refer to 
Appendix \ref{S.functoriality.MS} but I hope this to also serve as an introductory 
motivation for the following appendix. 
\end{Rem}


\appendix

\section{Morgan-Shalen type compactification}\label{Appendix.MS}

In this appendix, we discuss Morgan-Shalen compactifications \cite{MS}, 
in particular, its variants and extensions. They exist for fairly  
general ``spaces'' which do 
\textit{not} necessarily have good known 
modular interpretations. 
The original work \cite{MS} was later revisited by 
DeMarco-McMullen \cite{DMM}, Kiwi \cite{Kiwi}, 
Favre \cite{Fav} to relate to the Berkovich geometric context and 
then was partially extended by 
Boucksom-Jonsson \cite{BJ} more recently. Our intension of this appendix is 
to give natural further extensions of 
\cite{BJ} (hence \cite{MS} partially) to 
\textit{algebraic 
stacks} with some \textit{mild singularities} allowed, 
and then establish some basic properties. This appendix could be 
read for independent interest. 

The main purpose of our extension is that (later) we use such extensions for 
our studies of tropical geometric compactifications (cf., e.g., 
Theorem \ref{mg.hyb.tgc} and the remark at the end of \S Introduction) 
by comparison which 
will also continue in \cite{OO}. 
In particular, 
our extensions provide a language to describe 
our tropical geometric compactifications of $M_{g}$, $A_{g}$ (cf., also 
\cite{OO}). 

Some part of this appendix 
requires birational geometric jargon but interested readers who were 
not accustomed to such language, could assume the whole space 
$X$ to be an smooth orbifold and the boundary divisor $X\setminus U=D$ to be its 
simple normal crossing divisor. Indeed, it is the most important 
special case of both our dlt stacky pairs (to be introduced) and toroidal stacks. 
Indeed, for our practical applications, such case will be enough. 

\subsection{Slight extensions}

\subsubsection{Brief review of \cite{MS},\cite{BJ}}
\label{original.MSBJ}

We start with briefly recalling the original constructions of Morgan-Shalen and Boucksom-
Jonsson \cite{MS}, \cite[\S2]{BJ} in this section. In 1980s, 
Morgan-Shalen \cite{MS} constructed the compactifications of 
\textit{affine} complex varieties $U={\it Spec}(R)$ in terms of ring theory and valuations, 
depending on finite generators (as $\mathbb{C}$-algebra) of $R$. 
We refer to \cite[\S 1.3]{MS} for the details. They were motivated by studying 
the character varieties. 

Recently, 
Boucksom-Jonsson \cite{BJ} partially extended the construction to give 
compactifications of \textit{smooth} complex varieties 
$U$ as follows. 
Starting from algebraic compactification $U\subset X$, i.e., $X$ is a 
smooth proper variety with $D:=X\setminus U$ simple normal crossing divisor, 
they constructed a ``hybrid compactification" of $U$ as follows. 
We often denote it as $\bar{U}^{\it hyb}$ bravely instead of $\bar{U}^{\it hyb}(X)$ 
\footnote{Boucksom-Jonsson \cite{BJ} wrote this as 
$X^{\it hyb}$ in our notation.}
although it 
depends on $X$, in the case if $X$ is obvious from the context. 
Set theoritically, the compactification is simply 

$$\bar{U}^{\it hyb}:=U(\mathbb{C}) \sqcup \Delta(D),$$

\noindent 
where $\Delta(D)$ denotes the so-called ``dual (intersection) complex'' 
(also called the ``incidence complex''). See \cite{Kul, dFKX} for example. 
As in \cite{MS}, they used the 
logarithmic function to provide compact ``hybrid" topology to the above. 
The topology is characterised by the following: 
if $x\in D$ and local coordinates $f_{i} (i=1,\cdots, {\it dim}(X))$ at an 
open neighborhood satisfying $|f_{i}|<1$ for all $i$, then a 
sequence $x_{j} (j=1,2,\cdots)$ of $U$ converging to $x$, in turn 
converges in $\bar{U}^{\it hyb}$ to a point in $\Delta(D)$ 
with coordinates given by 
$(\cdots,\lim_{j\to \infty} \frac{\log|f_{i}(x_{j})|}{\log|\prod_{i}
(f_{i}(x_{j})|},\cdots).$ 
We refer to \cite[\S2]{BJ} for the details. 

\begin{Rem}
It is easy to see that the above Boucksom-Jonsson hybrid compactification 
\cite{BJ} for a log pair $U\subset X$ where $X$ is smooth and $X\setminus U$ 
a simple normal crossing divisor, 
satisfies the same property as our compactifications of 
$A_{g}$ and $M_{g}$ as proved in Corollaries 
\ref{val.cri.av} and \ref{val.cri.curve}. That is, 
if $(C\setminus \{p\})\to U$ is a holomorphic 
morphism from a punctured disk which extends 
holomorphically from $C$ to $X$, 
the original morphism also extends 
to a continuous map $C\to \bar{U}^{\it hyb}(X)$. 
\end{Rem}


\subsubsection{Extending to algebraic stack}\label{DMstack.MS}

All the discussions in this subsection works 
over general algebraically closed field $k$. 
We aim at extending the story to the category of stacks, and 
for that, we first give a natural set of stacky definitions as a 
preparation. We start with some obviously natural stacky extension, 
which was also discussed in the literatures for other purposes 
(cf., e.g., \cite[\S4]{Yas}). 

In this paper, \'etale chart of a Deligne-Mumford stack of finite type 
(over $k$) means 
an \'etale surjective morphism from a locally finite type scheme over $k$. 

\begin{Def}
A \textit{prime divisor} of 
a normal separated Deligne-Mumford stack 
$\mathcal{X}$ of finite type 
over $k=\mathbb{C}$ (DM stack, for short from now on) is a reduced 
closed substack of $\mathcal{X}$ of pure 
codimension $1$ 
which does not decompose as a union of proper closed substacks again of 
pure codimensions $1$. 
Such a divisor $\mathcal{D}$ is $\mathbb{Q}$-Cartier if its pull back to 
any \'etale chart is 
a (algebraically) $\mathbb{Q}$-Cartier divisor. It is easy to see that 
this condition does not depend on the charts. 
A $\mathbb{Q}$-divisor on $\mathcal{X}$ is a formal $\mathbb{Q}$-linear 
combination of prime divisors $\mathcal{D}_{1},\cdots,\mathcal{D}_{s}$ in the form 
$\sum_{1\le i\le s}a_{i}\mathcal{D}_{i}$ where all $a_{i}\in \mathbb{Q}$. 
\end{Def}

Discussions from the next definition \ref{stacky.log.pairs} 
until \ref{stacky.min.model} or \ref{dlt.stack.min.skeleton.conj} need to assume 
some acquaintance of the readers with the basic theory of the Minimal Model Program but 
for interested readers without it, one might be able to assume that being 
dlt is only slight extension of simple normal crossings, although very useful, 
invented by V.Shokurov. 

\begin{Def}\label{stacky.log.pairs}
We succeed the above notation. The pair 
$(\mathcal{X},\sum_{1\le i\le s}a_{i}\mathcal{D}_{i})$ is 
said to be a \textit{stacky log pair} if, for any \'etale chart 
$p\colon V\to \mathcal{X}$, the pair $(V,\sum_{i}a_{i} p^{*}\mathcal{D}_{i})$ 
is a log pair in the sense that $K_{V}+\sum_{i}a_{i} p^{*}\mathcal{D}_{i}$ 
is $\mathbb{Q}$-Cartier. The above pullback $p^{*}\mathcal{D}_{i}$ 
makes sense since $p$ is \'etale and it is straightforward to see that 
this condition does \textit{not} depend on the presentation $p$. 
\end{Def}

We remark that by the Keel-Mori theorem \cite{KeM}, 
we always have a coarse algebraic space $X$ of $\mathcal{X}$ and 
its primes divisors $D_{i}$ as coarse subspaces of $\mathcal{D}_{i}$s. 
If we take an \'etale cover $V\to \mathcal{X}$ and suppose the 
natural map $V\to X$ branches at prime divisors $B_{j}\subset X$ 
with order $m_{j}$, we call the pair 
$(X,D_{X}:=\sum_{i}a_{i}D_{i}+\sum_{j}\frac{m_{j}-1}{m_{j}}B_{j})$ 
the ``coarse pair'' of the stacky log pair 
$(\mathcal{X},\sum_{i}a_{i}\mathcal{D}_{i})$. 
By \cite[5.20]{KM98} for instance, this $(X,D_{X})$ is also a log pair in 
the sense the log canonical divisor is $\mathbb{Q}$-Cartier.

\begin{Def}
We succeed the above notation. The stacky log pair 
$(\mathcal{X},\sum_{1\le i\le s}a_{i}\mathcal{D}_{i})$ is said to be 
\begin{enumerate}
\item \textit{kawamata-log-terminal} if $(X,D_{X})$ is so. 

\item \textit{log canonical} if $(X,D_{X})$ is so. 

\end{enumerate}
\end{Def}

\begin{Def}
We succeed the above notation. The stacky log pair 
$(\mathcal{X},\sum_{1\le i\le s}a_{i}\mathcal{D}_{i})$ is said to be 
\textit{locally divisorially-log-terminal} or simply \textit{dlt stacky pair} 
for bravity, 
if there is an \'etale chart $p\colon V\twoheadrightarrow \mathcal{X}$ with 
$(V,\sum_{i}a_{i}p^{*}\mathcal{D}_{i})$ dlt with 
$\mathbb{Q}$-Cartier $p^{*}\mathcal{D}_{i}$s. 
\end{Def}

Here the above $\mathbb{Q}$-Cartierness again means the algebraic 
$\mathbb{Q}$-Cartierness on given normal variety $V$. 
The above notion, extending the (schematic) $\mathbb{Q}$-Cartier dlt 
pair, plays a central role in 
this appendix. Recall that an useful point of the concept of dlt comes from 
that all the 
lc centers inside the boundary of dlt pair are generically normal crossings as 
\cite[\S 3.9]{Fjn} shows (cf., also \cite[4.16]{Kol16}). 
However there is a subtlety that 
a log pair (in the category of varieties) being dlt stacky pair is 
not quite the same as dlt pair nor $\mathbb{Q}$-factorial dlt pair, first as the condition is only required 
\'etale locally 
and second for the 
$\mathbb{Q}$-Cartierness assumption of the boundaries. 
The coarse pair of dlt stacky pair around $0$-dimensional lc center 
with $\mathbb{Q}$-Cartier boundary components 
is called ``qdlt'' (quotient-dlt) 
in \cite{dFKX}. For our purposes, dual complex of the 
boundary at the coarse moduli space is not enough and 
essentially need stack structures as the following 
simple example shows (cf., also \cite[6.1.7]{ACP}). 

\begin{Ex}\label{MS.quot.ex}
Think of the quotient stack $[(\mathbb{A}_{x,y}^{2},(xy=0))/G]$, 
where $G:=\mathbb{Z}/2\mathbb{Z}$ 
acts on the affine plane by switching the coordinates i.e., $x\mapsto y, 
y\mapsto x$. Then the dual complex of the quotient is just one point 
while that of the stack is a segment divided by $\mathbb{Z}/2\mathbb{Z}$. 
(If one would like a compact example, then replace $\mathbb{A}^{2}$ simply by 
its projective compactification $\mathbb{A}^{2}\subset \mathbb{P}^{2}$.) 
\end{Ex}

\begin{Def}\label{stacky.min.model}
We succeed the above notation. A line bundle on a 
DM stack $\mathcal{X}$ is said to be \textit{nef} (resp., \textit{ample}) 
if it descends to a nef (resp., ample) $\mathbb{Q}$-line bundle on $X$. 
\end{Def}

\begin{Def}
We succeed the above notation. The stacky log pair 
$(\mathcal{X},\sum_{1\le i\le s}a_{i}\mathcal{D}_{i})$ is said to be 
\begin{enumerate}
\item \textit{stacky klt model} if it is stacky klt and 
$K_{\mathcal{X}}+\sum_{1\le i\le s}a_{i}\mathcal{D}_{i}$ is nef. 
\item \textit{stacky lc model} if it is stacky lc and 
$K_{\mathcal{X}}+\sum_{1\le i\le s}a_{i}\mathcal{D}_{i}$ is ample. 

\item \textit{stacky dlt model} if it is dlt and 
$K_{\mathcal{X}}+\sum_{1\le i\le s}a_{i}\mathcal{D}_{i}$ is nef. 
\end{enumerate}
\end{Def}

Please do not confuse stacky dlt pair and stacky dlt model (only the 
latter, which is the special cases of the former, 
requires the log-minimality condition). We can define the 
dual complex of any stacky dlt pairs as follows. 

\begin{DefProp}[Skeleta]\label{skeleton.dlt.stack}
For an arbitrary separated 
stacky dlt pair $(\mathcal{X},\mathcal{D}_{\mathcal{X}}=\sum_{i}\mathcal{D}_{i})$, 
we take an \'etale cover $p\colon V\to \mathcal{X}$ and set $W:=V\times_{\mathcal{X}}V$ with naturally induced morphisms $q_{i} (i=1,2) \colon 
W\to V$ (so that $\mathcal{X}=[W\rightrightarrows V]$). 
Now we consider the colimit of topological spaces 
$\Delta(\lfloor (p\circ q_{i})^{*}\mathcal{D}_{\mathcal{X}} \rfloor)
\rightrightarrows \Delta(\lfloor p^{*}\mathcal{D}_{\mathcal{X}} \rfloor)$, where $i=1,2$,  
$\Delta(-)$ denotes the dual complex (as in \cite{dFKX}) and the 
morphisms are affine linear at each simplex which extends the 
maps of vertices. We denote the colimit as topological space 
by $\Delta(\mathcal{D}_\mathcal{X})$ and call the \textit{dual (intersection) complex} or the 
\textit{skeleton} 
of the 
stacky dlt pair 
$(\mathcal{X},\mathcal{D}_{\mathcal{X}})$. 

Then, $\Delta(\mathcal{D}_{\mathcal{X}})$ does not depend on the choice of $p$ and 
hence well-defined which we call dual complex of the stacky dlt pair 
$(\mathcal{X},\mathcal{D}_{\mathcal{X}})$. \end{DefProp}

\begin{proof}
First we untangle the abstract definition of $\Delta(\mathcal{D}_\mathcal{X})$ as a 
cell complex in more concrete terms. Our topological colimit of $\Delta
(\lfloor (p\circ q_{i})^{*}\mathcal{D}_{\mathcal{X}} \rfloor)
\rightrightarrows \Delta(\lfloor p^{*}\mathcal{D}_{\mathcal{X}} \rfloor)$ 
 has an inductive ``skeleton"\footnote{in the context of cell complexes, 
 rather than that of Berkovich geometry} 
 structure, as being a cell complex, as follows. 
 It is simply because both 
  $\Delta
(\lfloor (p\circ q_{i})^{*}\mathcal{D}_{\mathcal{X}} \rfloor)$ 
and $\Delta(\lfloor p^{*}\mathcal{D}_{\mathcal{X}} \rfloor)$ 
have stratifications by the (inner parts of) $k$-skeleta and 
the two maps between them preserve the stratifications. 
Now, the $0$-skeleton $\Delta^{(0)}(\mathcal{D}_\mathcal{X})\subset 
\Delta(\mathcal{D}_\mathcal{X})$ 
is simply the colimit set (with discrete topology) of two maps 
$\Delta^{(0)}
(\lfloor (p\circ q_{i})^{*}\mathcal{D}_{\mathcal{X}} \rfloor)
\rightrightarrows \Delta^{(0)}(\lfloor p^{*}\mathcal{D}_{\mathcal{X}} \rfloor)$. 
Above $\Delta^{(0)}(-)$ simply denotes the sets of irreducible components 
(of each divisor $-$). We then proceed inductively as follows. 
Suppose we have constructed up to $(k-1)$-skeleton part ($k\in \mathbb{Z}_{>0}$) 
of $\Delta(\mathcal{D}_\mathcal{X})$ which we denote as 
$\Delta^{(k-1)}(\mathcal{D}_\mathcal{X})$. We write the 
set of codimension $k$ 
log-canonical centers of $\lfloor p^{*}\mathcal{D}_{\mathcal{X}} \rfloor$ 
(resp., $\lfloor (p\circ q_{i})^{*}\mathcal{D}_{\mathcal{X}} \rfloor$) 
as $C^{(k)}(\lfloor p^{*}\mathcal{D}_{\mathcal{X}} \rfloor)$ 
(resp., $C^{(k)}(\lfloor (p\circ q_{i})^{*}\mathcal{D}_{\mathcal{X}} \rfloor)$) and 
let $\tilde{\Delta}^k(\lfloor p^{*}\mathcal{D}_{\mathcal{X}} \rfloor)$ 
(resp., $\tilde{\Delta}^k(\lfloor (p\circ q_{i})^{*}\mathcal{D}_{\mathcal{X}} \rfloor)$)  
be defined as 
$$\bigsqcup_{S\in   C^{(k)}(\lfloor p^{*}\mathcal{D}_{\mathcal{X}} \rfloor)}
k-\text{simplex } \Delta_{S}$$

$$\biggl(\text{resp., }\bigsqcup_{S\in   C^{(k)}(\lfloor (p\circ q_{i})^{*}\mathcal{D}_{\mathcal{X}} \rfloor)}
k-\text{simplex } \Delta_{S}\biggr).$$ 
\noindent
Then, as the next step, we glue the topological colimit of the induced diagram 
$$\tilde{\Delta}^k(\lfloor p^{*}\mathcal{D}_{\mathcal{X}} \rfloor)
\rightrightarrows \tilde{\Delta}^k
(\lfloor (p\circ q_{i})^{*}\mathcal{D}_{\mathcal{X}} \rfloor)$$

\noindent
along the natural boundary map 
 $\partial\tilde{\Delta}^k
(\lfloor (p\circ q_{i})^{*}\mathcal{D}_{\mathcal{X}} \rfloor)
\to \Delta^{(k-1)}(\mathcal{D}_\mathcal{X})$. Note that the above maps are 
all cellular. 
Then we continue up to $k={\it dim}(\mathcal{X})$ so that 
the final outcome is nothing but our colimit $\Delta(\mathcal{D}_\mathcal{X})$ using 
the cover $V\twoheadrightarrow \mathcal{X}$. 

What we want to show is that the above 
$\Delta(\mathcal{D}_{\mathcal{X}})$ constructed 
via the chart $V$ does 
\textit{not} depend on the choice of $V$. 
Such independence assertion amounts to show the following: if $[W'\rightrightarrows 
V']$ is another presentation of $\mathcal{X}$ with an \'etale morphism 
$f\colon V'\to V$, then 
\begin{equation}\label{dual.cpx}
\begin{split}
&\Delta(\lfloor (p\circ r)^{*}\mathcal{D}_{\mathcal{X}} \rfloor)\\
\twoheadrightarrow&\Delta(\lfloor (p\circ f)^{*}\mathcal{D}_{\mathcal{X}} \rfloor)^{2}
\times_{\Delta(\lfloor p^{*}\mathcal{D}_{\mathcal{X}} \rfloor)
\times \Delta(\lfloor p^{*}\mathcal{D}_{\mathcal{X}} \rfloor)}
\Delta(\lfloor (p\circ q_{i})^{*}\mathcal{D}_{\mathcal{X}} \rfloor)
\end{split}
\end{equation}

\noindent
i.e., the above natural morphism is surjective, 
\footnote{It is \textit{not} injective in general which makes an obstacle to 
define the dual complex of algebraic stacks at \textit{topological stack} level 
for our general setting. See \cite[6.1.9. 6.1.10]{ACP} for related discussions. }
where
$r\colon (V'\times V')\times_{(V\times V)}W \cong W'\to V$ denotes 
the naturally induced morphism. 
Also note that since 
$q_{1}^{*}\mathcal{D}_{\mathcal{X}}=q_{2}^{*}\mathcal{D}_{\mathcal{X}}$ as 
$\mathcal{D}_{\mathcal{X}}$ is a stacky divisor, the right hand side 
of (\ref{dual.cpx}) is independent of $i$. 
To prove the above required surjectivity 
(\ref{dual.cpx}) at the level of $k$-skelta by induction on $k$ is fairly stratighforward as 
follows. First, such assertion for the $k=0$ case is surjectivity of the natural map 
\begin{equation}\label{centers.map}
\begin{split}
&C^{(0)}(\lfloor (p\circ r)^{*}\mathcal{D}_{\mathcal{X}} \rfloor)\\
\twoheadrightarrow
&C^{(0)}(\lfloor (p\circ f)^{*}\mathcal{D}_{\mathcal{X}} \rfloor)^{2}
\times_{(\Delta^{(0)}(\lfloor p^{*}\mathcal{D}_{\mathcal{X}} \rfloor))^{2}}
C^{(0)}(\lfloor (p\circ q_{i})^{*}\mathcal{D}_{\mathcal{X}} \rfloor). 
\end{split}
\end{equation}

This holds immediately as 
$(q_{1}\times q_{2})\colon W\to V\times V$ is \'etale to its image and 
$f$ is also \'etale. 
Suppose that we know (\ref{dual.cpx}) up to $(k-1)$-skeleta 
level. Then we want to show that the $k$-dimensional cells canonically coincides 
between the both hand sides of (\ref{dual.cpx}). This is nothing but the 
same claim as (\ref{centers.map}) above also holds when we replace $0$ by $k$
but, by definition, it is straightforward by the same reason that 
$q_{1}\times q_{2}$ and $f$ are \'etale at open neighborhoods of 
generic points of codimension $k$ log-canonical centers. 
\end{proof}

The above obviously 
extends the construction in schematic case (cf., e.g., \cite{dFKX,NX}) 
and also coincides with \cite[\S 6]{ACP} when overlaps. 
In particular, note that the above defined dual complex is \textit{not} the same as 
the dual complex of the coarse pair. Indeed, the previous 
Ex. \ref{MS.quot.ex}, \cite[6.1.7]{ACP} provide simple counterexamples.


It is natural to expect that roughly speaking the dual complex of ``minimal model'' 
does not depend on the choice. More precisely, we conjecture the following after 
\cite{dFKX} which establish its some versions for schematic case. 

\begin{Conj}[Minimal skeleton]
\label{dlt.stack.min.skeleton.conj}
Once we fix a (kawamata-)log terminal DM stack $\mathcal{U}$, 
then the homeomorphic type of the dual complex of stacky dlt model 
$(\mathcal{X},\sum_{1\le i\le s}\mathcal{D}_{i})$ 
with $\mathcal{X}\setminus \sum_{i}\mathcal{D}_{i}=\mathcal{U}$, 
does not actually depend on the choice of such compactifications. 
\end{Conj}

\begin{DefProp}[Compactifications]\label{MS.dlt.stack}
We keep the notation of \ref{skeleton.dlt.stack}. 
Then for the open substack 
$\mathcal{U}:=\mathcal{X}\setminus {\it Supp}(\mathcal{D}_{\mathcal{X}})$ and 
its coarse moduli space $U$, we can construct a Morgan-Shalen-Boucksom-Jonsson 
partial compactification 
$\bar{U}^{\it hyb}(\mathcal{X}):=U\sqcup \Delta(\lfloor \mathcal{D}_{\mathcal{X}}\rfloor)$ with a Hausdorff topology extending the complex analytic topology of 
$U(\mathbb{C})$. 
If $\mathcal{X}$ is proper and $\mathcal{D}_{\mathcal{X}}$ is a (effective)  
$\mathbb{Z}$-divisor, then $\bar{U}^{\it hyb}(\mathcal{X})$ is also compact. 
\end{DefProp}

\begin{proof}
We simply imitate the construction of Boucksom-Jonsson \cite{BJ} 
which we reviewed at \S \ref{original.MSBJ}. 
We write for the preimage of $U$ to $V$ (resp., $W$) as $U_{V}$ (resp., $U_{W}$). 

Then we first construct $\bar{U_{V}}^{\it hyb}(V)$ 
essentially following the method of \cite[\S2]{BJ} as follows. 
For each point $x\in V\setminus U_{V}$, 
consider the log-canonical center $Z$ which includes $x$. By 
the definition of stacky dlt pair,  
if we replace $V$ by 
sufficiently small open subset $V_{x}$ of $V$, 
we can and do assume 
such log canonical 
center is the intersection of $m$ 
$\mathbb{Q}$-Cartier boundary divisor $D_{V_{x},i} (i=1,\cdots,m)$ and 
for sufficiently divisible $l_{i}\gg 0$,  $l_{i}D_{V,i}$ is Cartier so that 
they can be written as 
$f_{i}=0$ by some holomorphic $f_{i}$ on $V_{x}$. 
We shrink $V_{x}$ small enough to the locus $|f_{i}|<1$ 
if necessary. 
By running all such $x$ and 
using these $f_{i}$s on each $V_{x}$, 
we can construct the partial compactification of $V_{x}$ 
as $\overline{(V_{x}\setminus (\cup_{x,i}D_{V_{x},i}))}^{\it hyb}(V_{x})$ 
completely similarly as \cite[\S 2.2]{BJ}. Then, from our constructions,  
$\{\overline{(V_{x}\setminus (\cup_{x,i}D_{V_{x},i}))}^{\it hyb}(V_{x})\}_x$ 
naturally glue together to form $\bar{U_{V}}^{\it hyb}(V)$. 
In the same way, we also get $\bar{U_{W}}^{\it hyb}(W)$. 

Now, we construct $\bar{U}^{\it hyb}(\mathcal{X})$
as the (topological) colimit of the natural diagram 
$\bar{U_{W}}^{\it hyb}(W)\rightrightarrows 
\bar{U_{V}}^{\it hyb}(V)$. Independence of such colimit from the cover $V$ is 
proved completely similarly as the above proof of \ref{skeleton.dlt.stack}, thus 
we omit the details. 
\end{proof}

\begin{Ex}[Torus embedding case description]\label{toric.MS}

We give a description for toric case. What we means is the following. 
Starting from arbitrarily proper torus embedding $T\subset X$, 
we can do toric log resolution of $(X,X\setminus T)$; $f\colon \tilde{X}\to 
X$. Then with trivial stack structure, we can talk about 
the dual complex of $\tilde{X}\setminus T$ and corresponding 
Morgan-Shalen-Boucksom-Jonsson compactification of $T(\mathbb{C})$. 
Then we can 
concretely see the resulting compactifications indeed 
do not depend on the (complete) fan structure. 
See \cite{KKMS} for the basics of the toric (or toroidal) geometry. 
Let $N\cong \mathbb{Z}^{n}$ be a lattice and $T:=T_{N}:=N\otimes_{\mathbb{Z}}\mathbb{G}_{m}$ be the associated algebraic torus. 
We take a basis of $N$ so that we sometimes identify $N$ as $\mathbb{Z}^{n}$ and $T$ as $(\mathbb{G}_{m})^{n}$. 
We consider the tropicalization map 
\footnote{it can be also seen as the moment map with respect to the 
$(S^{1})^{n}(\subset T)$-action and the K\"ahler form 
$\prod_i \frac{dz_{i}\wedge {\bar dz_i}}{|z_i|^2}.$}

$$m\colon T(\mathbb{C}) \twoheadrightarrow N_{\mathbb{R}}:=N\otimes_{\mathbb{Z}}\mathbb{R},$$

\noindent
which is, via the basis of $N$, written as $$(z_{1},\cdots,z_{n})\mapsto (-{\it log}|z_{1}|,-{\it log}|z_{2}|,\cdots,-{\it log}|z_{n}|).$$ 
It is easy to verify that this definition does not depend on the choice of basis of $N$. 
This logarithmic mapping is a key in the construction of \cite{BJ}. 

It is natural to attach the infinite hyperplane to form the natural 
projective compactification $N_{\mathbb{R}}\subset \mathbb{P}_{N_{\mathbb{R}}}=:\mathbb{P}=N_{\mathbb{R}}\sqcup 
((N_{\mathbb{R}}\setminus \{0\})/\mathbb{R}^{*})$. We also denote the boundary 
$((N_{\mathbb{R}}\setminus \{0\})/\mathbb{R}^{*})$ as $\partial \mathbb{P}$. 
Note this compactification is different from 
$N_{\mathbb{R}}\cong \mathbb{R}^{n}\subset (\mathbb{R}\sqcup \{+\infty \})^{n}$ relative to the interior. 

From the compactness of the real projective space $\mathbb{P}$, 
for any sequence $x_{i} (i=1,2,\cdots) \in T$, 
after passing to subsequence if necessary, $m(x_{i})$ converge to some point in $
\mathbb{P}$. 
Then, the natural hybrid 
compactification for a torus embedding 
$T\subset X$ with simple normal crossing toric boundary $X\setminus T$ 
by \cite{BJ} can be reconstructed as 

$$T\sqcup ((N_{\mathbb{R}}\setminus \{0\})/\mathbb{R}^{*}).$$

\noindent
The natural topology we put on the above space as 
an extension of the complex analytic topology on $T$ is defined by 
the convergence of sequences (or nets) of points $x_{i}$ of $T$, 
which does \textit{not} converge inside $T$, to a point of boundary $\partial 
\mathbb{P}$ as 
the convergence of the sequence $m(x_{i})$. 

It is easy to see that this is equivalent to the definition of \cite{BJ} 
(cf., also \ref{original.MSBJ}), 
thus independent of the choice of above toric compactification. 

Here is the local description. 
For an affine toric variety $(T\subset) X=U_{\sigma}={\it Spec}(\mathcal{S}_{\sigma})$, we simultaneously imitate and extend the above. 
Each $m\in \mathcal{S}_{\sigma}$ corresponds to a regular function $e(m)$ on $U_{\sigma}$ of monomial type and 
we take a finite subset $S\subset \mathcal{S}_{\sigma}$ which generates the function ring $\mathcal{S}_{\sigma}$. 
Then we define the moment map 

$$
m_{S}\colon X\to \mathbb{R}^{S}
$$

\noindent 
as $x\mapsto (-log|e(m)(x)|)_{m\in S}$. It is standard 
to see that this is nothing but the combination of 
surjection $X\twoheadrightarrow X/CT$, where $CT:=N\otimes U(1)$ is the natural 
compact form of $T$, followed by 
the well-known topological embedding of 
$X/CT$ into a manifold with corners (cf., \cite{Oda88}). 
Then finally we consider 

$$\partial^{\sigma} X:=\{{\it lim}_{i}m_{S}(x_{i})\in ((N_{\mathbb{R}}\setminus \{0\})/\mathbb{R}^{*}) \mid x_{i}(i=1,2,\cdots) \text{ converges in } U_{\sigma}  \},$$

\noindent
and set $\bar{X}^{\sigma}:=X\sqcup \partial^{\sigma} X$ with natural topology defined by the above convergence. 
It follows straightforward from the construction that $X$ is open dense inside 
$\bar{X}^{\sigma}$. 
It is also easy to see that the partial compactification $X\subset \bar{X}^{\sigma}$ does not depend on the choice of $S$, and furthermore that 
as far as $S$ with the origin spans the maximal ($n$-) dimensional space, 
the above construction works and gives the same outcome $\bar{X}^{\sigma}$. 
Our construction of this $\bar{X}^{\sigma}$ is a special case of \cite[\S I.3]{MS}. 

\end{Ex}


\begin{Ex}[For toroidal stacks \cite{ACP}]\label{toroid.MS}
Here we follow \cite{ACP} and \cite{Thu}. 
Although not all toroidal DM stack (cf., \cite[6.1.1]{ACP}) form 
dlt stacky pairs as not all toric singularities are dlt, \cite{Thu, ACP} 
nevertheless gives a 
natural partial generalization of the dual complex construction to toroidal embedding \textit{stack} $\mathcal{U}\subset \mathcal{X}$ 
to form $\Delta(\mathcal{X}\setminus \mathcal{U})$  (especially \cite[\S 6]{ACP}) 
extending the skeleton of 
Thuillier \cite{Thu}. Note that this construction \cite{Thu, ACP} 
essentially uses the natural log structure and indeed it is 
further extended to general fine saturated log schemes by \cite{Uli}. 
Let us briefly review the construction and 
give a corresponding Morgan-Shalen type compactification 
from a perspective of the previous discussion \ref{toric.MS}. 
The resulting compactification will be denoted as 
$\bar{U}^{\it hyb}:=U\sqcup (\Delta(\mathcal{X}\setminus \mathcal{U}))$, 
where $U$ is the coarse moduli space of $\mathcal{U}$. 

Let us take an \'etale cover $p\colon V\to \mathcal{X}$, denote the preimage of 
$\mathcal{U}$ as $U_{V}$, and set $W:=V\times_{\mathcal{X}}V$ as before. 
Then we put $D_{V}:=V\setminus U_{V}$. 
For each point $x\in D_{V}$, we can take an euclidean open neighborhood 
of $x\in U_{V,x}
\subset V$ which has analytic isomorphism to 
an euclidean open neighborhood of a point in the boundary of some affine toric variety $V_{x}$. We suppose that 
all the boundary components intersect the open subset $U_{V,x}$ and denote the isomorphism as 
$i_{x}\colon U_{V,x}\hookrightarrow V_{x}$. 

Denote the corresponding cone of $V_{x}$ by $\sigma_{x}$ and set the dual $\mathcal{S}_{\sigma_{x}}:=\{m\in M:={\it Hom}_{\mathbb{Z}}(N,\mathbb{Z})\mid 
(m,n)\ge 0 \text{ for all }n\in \sigma_{x} \}$ so that $V_{x}={\it Spec}(\mathbb{C}[\mathcal{S}_{\sigma_{p}}])$. 
Each $m\in M$ corresponds to a function $e(m)$ on $X_{x}$ so $i_{x}^{*}(e(m))$ on $U_{V,x}$. 
Take a finite generating system $\{m_{i}\}_{i\in S}$ of the semigroup $\mathcal{S}_{\sigma_{x}}$ and 
exploits the partial compactification in the previous section i.e. 
we consider $\partial^{\sigma_{x}} X_{x}$ and correspondingly we take partial compactification of $V_{x}$ which we denote by $\bar{V}^{\{x\}}$. 
It is straightforward to see that $\bar{V}^{\{x\}}$ glues together to 
form a partial compactification $\bar{V}$. 
Indeed, if we take another isomorphism $i'_{p}\colon U_{x}\cong V_{x}\subset X_{x}$, by shrinking $U_{x}$ if necessary, 
$\frac{(i'_{x})^{*}(e(m))}{i_{x}^{*}(e(m))}$ are non-vanishing well-defined function on $U_{x}$ 
for any $m\in M$ (we can check this by restricting to finite generators). 
Similarly we can do the same construction to form a partial compactification 
$\bar{W}$ of $W$. 
As in the previous \ref{MS.dlt.stack}
we define the 
desired generalized hybrid compactification of $U$ 
as the colimit of $\bar{W}\rightrightarrows \bar{U}$. 
Independence of the construction from $V$ is proved completely similarly 
as \ref{MS.dlt.stack} so we avoid to repeat the details of its proof. 
\end{Ex}

\begin{Rem}[{\cite{OO}}]\label{toroidal.compactification.MS}

For toroidal compactifications of locally Hermitian symmetric space \cite{AMRT}, 
we can also naturally assign 
hybrid compactification as either special case of 
Definition-Proposition \ref{MS.dlt.stack} (when it is smooth stack with 
normal crossing boundary) or 
Example \ref{toroid.MS} in general. It is straightforward from the constructions that 
it does not depend on the admissible cone decompositions. 
Then, in a forthcoming paper \cite{OO}, we showed that first it does not 
depend on the admissible cone decompositions and 
such compactification for $A_g$ case coincides with 
our $\bar{A_g}^{T}$. 
\end{Rem}



\subsubsection{About ``gluing function''}
\label{MS.gen.fun}

This subsection means to be a simple remark that 
the logarithmic function used in \cite{BJ} for the hybrid compactification  
can be replaced by more general diverging function $f$, which we would call 
\textit{glueing function}. Here is the condition for such functions to be 
used: 

\begin{quote}
$f\colon D^{*}(\epsilon)\to \mathbb{R}_{>0}$ is a continuous function from 
$D^{*}(\epsilon):=\{z\in \mathbb{C} \mid 0<|z|<\epsilon \}$ for 
$0<\epsilon\ll 1$ such that 

\begin{enumerate}

\item $f(z)\to +\infty$ when $z\to 0$, 

\item for any $c\in \mathbb{C}^{*}$ $f(cz)-f(z)=O(1)$ when $z\to 0$. 

\end{enumerate} 
\end{quote}

The above condition morally tells that the function 
grows not (asymptotically) faster than the logarithmic function. 
Indeed, it is straightforward to see that all our constructions of 
Morgan-Shalen-Boucksom-Jonsson partial compactifications only use the above 
properties. 
Thus, our 
generalized hybrid compactification is similarly defined as 
$$\bar{U}^{\it hyb}_{(f)}=\bar{U}^{\it hyb}_{(f(z))}:= 
U(\mathbb{C})\sqcup \Delta(D),$$ 
\noindent 
i.e., set-theoricially same as Boucksom-Jonsson hybrid space, 
with \textit{modified hybrid topology} depending on $f$, but defined 
just by imitating \cite{BJ} which was the case $f(z)=-{\it log}|z|$. 
Hence, $\bar{U}^{\it hyb}_{({\it log}|z|)}=\bar{U}^{\it hyb}$ and so 
we tend to omit the subscript $_{(f)}$ when $f$ is the usual logarithmic function 
as above. Note that in our proof of 
Proposition \ref{MS.dlt.stack}, the logarithmic function 
can be replaced by any function satisfying above so that we can define 
$(U\subset)\bar{U}^{\it hyb}_{(f)}$ for stacky dlt pair 
$(\mathcal{X},\sum_{i}\mathcal{D}_{i})$. 

\subsection{Functoriality of MSBJ construction}
\label{S.functoriality.MS}

\begin{Thm}[Functoriality]\label{functoriality.MS}
The skeleta and the Morgan-Shalen-Boucksom-Jonsson partial compactifications  
(\cite{MS},\cite{BJ},\ref{skeleton.dlt.stack}, \ref{MS.dlt.stack},\ref{toroid.MS},\ref{MS.gen.fun}) 
are both functorial in the 
following sense. 

Suppose 
$(\mathcal{X},\mathcal{D}_{\mathcal{X}})$ and 
$(\mathcal{Y},\mathcal{D}_{\mathcal{Y}})$ are 
dlt stacky pairs (resp., 
$(\mathcal{X}\setminus\mathcal{D}_{\mathcal{X}})\subset \mathcal{X}, 
(\mathcal{Y}\setminus\mathcal{D}_{\mathcal{Y}})\subset \mathcal{Y}$ are 
toroidal DM stacks). If 
$f\colon \mathcal{X}\to 
\mathcal{Y}$ is a representable morphism (resp., 
toroidal morphism) such that $f^*\mathcal{D}_{\mathcal{Y}}=
\mathcal{D}_{\mathcal{X}}$. 
Denote the coarse moduli of $\mathcal{X}\setminus \mathcal{D}_{\mathcal{X}}$ 
(resp., $\mathcal{Y}\setminus \mathcal{D}_{\mathcal{Y}}$)
by $U_{\mathcal{X}}$ (resp., $U_{\mathcal{Y}}$). 
Then the induced map $U_{\mathcal{X}}\to U_{\mathcal{Y}}$ with complex analytic 
topologies 
continuously extends in a unique way to 
$\bar{U}_{\mathcal{X}}^{\it hyb}\to \bar{U_{\mathcal{Y}}}^{\it hyb}$. 
Moreover, if $(\mathcal{X},\mathcal{D}_{\mathcal{X}})$ and 
$(\mathcal{Y},\mathcal{D}_{\mathcal{Y}})$ are both dlt and 
toroidal, with $f$ toroidal morphism, then the two constructions coincide. 
For this theorem \ref{functoriality.MS}, 
the glueing function needs to be the usual logarithmic function. 
\end{Thm}

\begin{proof}
First we see that dual complexes are functorial. 
Passing to an \'etale cover, 
we can and do assume $\mathcal{X}$ and $\mathcal{Y}$ are varieties. 
(The necessary arguments for such reduction is again in the same way as 
our \S \ref{DMstack.MS} so we omit the details. )
We create a natural map at the $k$-skeleta level of 
$\Delta(\mathcal{D}_{\mathcal{X}})$ and $\Delta(\mathcal{D}_{\mathcal{Y}})$ 
on induction on $k$. The $k$-simplices forming $\Delta(\mathcal{D}_{\mathcal{X}})$ 
corresponds to codimension $k$ lc centers of $(\mathcal{X},\mathcal{D}_{\mathcal{X}})$. If we take a general point $x$ of an arbitrary lc center $Z$ 
and take a sufficiently small euclidean open neighborhood $O_{x}$ of $x$, 
$Z\subset O_{x}$ can be written as $(z_{1}=\cdots=z_{k}=0)$ where 
$z_{i} (i=1,\cdots,{\it dim}(X))$ are local holomorphic functions of 
$O_{x}$. If we suppose that the lc center of $(\mathcal{Y},\mathcal{D}_{\mathcal{Y}})$ is $l$-dimensional $Z'$, we can 
take local holomorphic functions around $f(x)$ as 
$w_{1},\cdots,w_{{\it dim}(Y)}$ with 
$(\prod_{1\le i\le l} w_{i}=0)\cap U_{\mathcal{Y}}=Z'$. From the assumptions, 
we can write $f$ as $w_{i}=\prod_{1\le j\le k}z_{j}^{m_{i,j}}\cdot g_{i,j}(\vec{z})
$ for all $i=1,\cdots,l$ with some invertible functions $g_{i,j}$. 
(It morally says that $f$ is not so far from monomial maps.) 
Furthermore, as $f^{*}\mathcal{D}_{\mathcal{Y}}=\mathcal{D}_{\mathcal{X}}$, 
$\sum_{i}m_{i,j}>0$ for any $j$ and $\sum_{j}m_{i,j}>0$ for any $i$, hence 
$k\ge l$ in particular. 
The matrix 
$(m_{i,j})_{i,j}$ induces a morphism from the $k$-simplex corresponding to $Z$ to 
a $l$-simplex corresponding to $Z'$. It naturally glues to form a continuous 
map $\Delta(\mathcal{D}_{\mathcal{X}})\to  \Delta(\mathcal{D}_{\mathcal{Y}})$, 
which does not depend on the choice of 
$(U_{\mathcal{X}},z_{j})$ and $w_{i}$s. 
From the above construction, it is also obvious that 
the map $U_{\mathcal{X}}\sqcup \Delta(\mathcal{D}_{\mathcal{X}})\to 
U_{\mathcal{Y}}\sqcup \Delta(\mathcal{D}_{\mathcal{Y}})$, 
which is simply obtained as a disjoint union of the two maps, 
is continuous. \end{proof}

By applying the above Theorem \ref{functoriality.MS} 
to the extended Torelli maps 
$\bar{\mathcal{M}_{g}}\to \bar{\mathcal{A}_{g}}$ (\cite{Nam2, Ale2} etc) 
we will get 
more counterexamples systematically to the continuity of ``glued" Torelli map 
$t_{g}$ we discussed around Theorem \ref{Torelli.Not}. 


\vspace{5mm} \footnotesize \noindent
Contact: {\tt yodaka@math.kyoto-u.ac.jp} \\
Department of Mathematics, Kyoto University, Kyoto 606-8285. JAPAN \\

\end{document}